\documentclass[12pt]{amsart}
\usepackage{tkz-euclide}
\usepackage{latexsym}
\usepackage{amssymb,amsmath,mathrsfs}
\usepackage{cases}
\usepackage{amsthm}
\usepackage{hyperref}
\usepackage{enumitem}
\usepackage[capitalise]{cleveref}
\usepackage{mathtools}

\makeatletter
\newenvironment{mcases}[1][l]
 {\let\@ifnextchar\new@ifnextchar
  \left\lbrace
  \array{@{}l@{\quad}#1@{}}}
 {\endarray\right.}
\makeatother

\allowdisplaybreaks 
\newtheorem{theorem}{Theorem}[section]
\newtheorem{definition}{Definition}[section]

\newtheorem{corollary}[theorem]{Corollary}
\theoremstyle{definition}

\makeatletter
\def\th@remark{%
  \thm@headfont{\bfseries}%
  \normalfont 
  \thm@preskip\topsep \divide\thm@preskip\tw@
  \thm@postskip\thm@preskip
}
\makeatother

\theoremstyle{remark}

\newtheorem{lemma}[theorem]{Lemma}

\newcommand{\Hilbert}{\mathbb{H}}

\newcommand{\half}{(-\overline{\Delta})^\frac{1}{2}}

\newcommand{\fractional}{(-\overline{\Delta})^\alpha}

\newcommand{\grad}{\overline{\nabla}}
\newcommand{\lap}{\overline{\Delta}}

\newcommand{\dnu}{\partial_\nu \Psi}

\newcommand{\C}{\mathcal{C}}
\newcommand{\lambdan}{\sqrt{\lambda_n}}

\def\XXint#1#2#3{{\setbox0=\hbox{$#1{#2#3}{\int}$ }
\vcenter{\hbox{$#2#3$ }}\kern-.6\wd0}}

\newcommand{\Rthreeplus}{\mathbb{R}^3_+}

\title{The Inviscid 3D Quasi-Geostrophic System on Bounded Domains}
\author[Novack]{Matthew D. Novack}
\author[Vasseur]{Alexis F. Vasseur}
\address[Matthew D. Novack]{\newline Department of Mathematics, \newline The University of Texas at Austin, Austin, TX 78712, USA}
\email{mnovack@math.utexas.edu}
\address[Alexis F. Vasseur]{\newline Department of Mathematics, \newline The University of Texas at Austin, Austin, TX 78712, USA}
\email{vasseur@math.utexas.edu}

\date{\today}

\thanks{\textbf{Acknowledgment}: The second author was partially funded by the NSF during this work}

\begin{document}
\begin{abstract}
We present a formal derivation of the inviscid 3D quasi-geostrophic system (QG) from primitive equations on a bounded, cylindrical domain.  A key point in the derivation is the treatment of the lateral boundary and the resulting boundary conditions it imposes on solutions.  To our knowledge, these boundary conditions are new and differentiate our model from closely related models which have been the object of recent study. These boundary conditions are natural for a variational problem in a particular Hilbert space. We construct solutions and prove an elliptic regularity theorem corresponding to the variational problem, allowing us to show the existence of global weak solutions to (QG).
\end{abstract}
\maketitle \centerline{\date}
\section{Introduction}
In this paper we study the inviscid three-dimensional quasi-geostrophic system.  The QG model describes stratified flows on a large time scale for which the effect of the rotation of the Earth is significant. The model consists of two coupled transport equations as follows:
\[\begin{dcases}
        \left( \partial_t + \grad^\perp \Psi\cdot \grad \right) \left( \mathcal{L}(\Psi) + \beta_0 y \right)= a_L & \Omega\times[0,h]\times[0,T] \\
        \left( \partial_t + \grad^\perp \Psi \cdot \grad \right) (\partial_{\nu}\Psi) = a_\nu & \Omega\times\{0,h\}\times[0,T].\\
       \end{dcases}   \qquad (QG)
\]
Classically, the model is posed for $\Omega = \mathbb{R}^2$ or $\Omega=\mathbb{T}^2$.  We use the notation 
$$ \grad = ( \partial_x, \partial_y, 0 ), \qquad \grad^\perp = ( -\partial_y, \partial_x ,0) .$$
The functions $a_L$ and $a_\nu$ are forcing terms, and $\beta_0$ is a parameter coming from the usual $\beta$-plane approximation. The normal derivative of $\Psi$ on $\Omega\times\{0,h\}$ is denoted by $\dnu$. The operator $\mathcal{L}$ is defined by
$$ \mathcal{L}:=\partial_{xx}+\partial_{yy}+ \partial_z \left( \lambda \partial_z \right) $$
where $\lambda>0$ is a smooth function depending only on $z$ and is related to the density of the fluid. To ensure ellipticity of $\mathcal{L}$ we require $$\frac{1}{\Lambda} \leq \lambda(z) \leq \Lambda$$ for some $\Lambda \in (0,\infty)$. Throughout the remainder of the paper, the system shall be posed on a fixed cylindrical domain 
$$ \Omega\times[0,h] $$
where $\Omega\subset \mathbb{R}^2$ is a smooth, bounded set, and the height $h$ is fixed and finite. The values of $\mathcal{L}(\Psi)$ and $\dnu$ are advected by the fluid velocity field $\grad^\perp\Psi$. In order to reconstruct $\Psi$ at each time, it is necessary to supplement the system with a boundary condition on the lateral boundary $\partial\Omega \times [0,h]$. 

\subsection{Boundary Conditions}

The purpose of this paper is to formally derive an appropriate model from the primitive equations while assuming that the lateral boundary is \textit{impermeable}; that is, we assume only that the fluid velocity $\grad^\perp\Psi$ is tangent to $\partial\Omega\times[0,h]$. We then prove that weak solutions exist globally in time for the resulting system.  In fact, we show in \cref{bdrysection} that the impermeability produces two constraints on a possible solution.  First, we must have that
\begin{align}
\Psi(t,x,y,z)|_{\partial\Omega\times[0,h]} = c(t,z) \label{bdrycond1}
\end{align}
for some unknown function $c(t,z)$.  However, this is not enough to define a unique solution to an elliptic problem on $\Omega\times[0,h]$.  Crucially, the impermeability condition provides another natural constraint.  After defining $\nu_s$ to be the normal derivative to $\partial\Omega\times\{z\}$ and $\,d\omega$ the Hausdorff measure on $\partial\Omega$, the second constraint is that for all $z\in[0,h]$,
\begin{align}
\frac{\partial}{\partial t}\int_{\partial\Omega\times\{z\}}  \grad\Psi\cdot \nu_s \,d\omega = 0.  \label{bdrycond2}
\end{align}
In other words, building a weak solution to (QG) requires choosing a datum $j_0(z):[0,h]\rightarrow\mathbb{R}$ such that for all time,
$$ \int_{\partial\Omega\times\{z\}}  \grad\Psi(t)\cdot \nu_s \,d\omega = j_0(z). $$
These two conditions differentiate the model we derive from closely related models which have been studied recently by Constantin and Nguyen \cite{cn}, \cite{cn2} and Constantin and Ignatova \cite{ci}, \cite{ci2}.  While we shall explain this distinction in detail in \cref{thesection}, we first describe a rough sketch of our existence proof, and then state our main results.

In \cite{pv} and \cite{novackweak}, the authors used the observation that the transport equations for $\mathcal{L}(\Psi)$ and $\dnu$ in (QG) formally preserve the norms of the data for an elliptic problem with Neumann boundary condition. Therefore, a sequence of approximate solutions $\Psi_n$ for which $\mathcal{L}(\Psi_n)$ and $\partial_{\nu}\Psi_n$ converge weakly in (respectively) $L^\infty_t(L^2(\Omega\times[0,h]))$ and $L^\infty_t(L^2(\Omega\times\{0,h\}))$ will have strong convergence for $\nabla\Psi_n$ in $L_t^\infty(L^2(\Omega\times[0,h]))$.  A key property of the (QG) system is a reformulation of the system in terms of $\nabla\Psi$.  This reformulation, first utilized extensively by Puel and the second author in \cite{pv}, can be seen at the level of the primitive equations and draws an analogy to the parallel formulations of the 3D Euler equations in terms of the velocity and the vorticity.  Unlike Euler, however, the strong convergence then allows one to pass to the limit at the level of $\nabla\Psi_n$ to construct a weak solution.  In the setting of the bounded domain $\Omega\times[0,h]$, it is not immediate that imposing \eqref{bdrycond1} and \eqref{bdrycond2} on the lateral boundary will allow for compactness at the level of $\nabla\Psi_n$ in $L_t^\infty(L^2(\Omega\times[0,h]))$.  Indeed, it might seem possible that because \eqref{bdrycond2} only controls the average of $\grad\Psi \cdot \nu_s$ on the sides, $\grad\Psi\cdot\nu_s$ could oscillate quite badly on $\partial\Omega\times[0,h]$.  To address this, we must formulate \eqref{bdrycond2} weakly (see \cref{weakbdrycond2} in Section 3). However, we also prove an elliptic regularity theorem (\cref{superduperelliptic}) which implies that in fact $\grad\Psi \cdot\nu_s \in L^2(\partial\Omega\times[0,h])$ is well-defined \textit{pointwise}, and $\nabla\Psi_n$ converges strongly to $\nabla\Psi$ in $L_t^\infty(L^2(\Omega\times[0,h]))$. To the authors’ knowledge, this type of boundary condition and the corresponding elliptic regularity theorem are novel.

\subsection{Main Result}
Before stating the existence theorem, we must provide several definitions.  The first is a natural compatibility condition between the elliptic operator and boundary conditions.  
\begin{definition}\label{compatibility}
Any triple $(f,g,j)$ of functions with $f(x,y,z)\in L^2(\Omega\times[0,h])$, $g(x,y,z)\in L^2(\Omega\times\{0,h\})$, $j(z)\in L^2(0,h)$ is compatible if
$$\int_{\Omega\times[0,h]}f(x,y,z) \,dx\,dy\,dz = \int_0^h j(z)\,dz + \int_{\Omega\times\{0,h\}} \lambda(z)g(x,y,z) \,dx\,dy.$$
A pair $(a_L, a_\nu)$ of forcing terms is compatible if $a_L \in L^1\left([0,T]; L^2(\Omega\times[0,h])\right)$, and $a_\nu \in L^1\left([0,T]; L^2(\Omega\times\{0,h\})\right)$ for all $T>0$ with 
$$ \int_{\Omega\times[0,h]}a_{L}(x,y,z)\,dx\,dy\,dz = \int_{\Omega\times\{0,h\}}\lambda(z) a_\nu(x,y,z) \,dx\,dy $$
\end{definition}
Next, we define the notion of weak solutions to the transport equations in (QG).
\begin{definition}\label{weaktransport}
Let $T>0$ be given and $\Psi(t,x,y,z):[0,T]\times\Omega\times[0,h]\rightarrow \mathbb{R}$ be such that $\nabla\Psi, \mathcal{L}(\Psi)\in L^\infty([0,T];L^2(\Omega\times[0,h]))$, $\partial_\nu\Psi\in L^\infty\left( [0,T];L^2(\Omega\times\{0,h\}) \right) $.  Then $\Psi$ is a weak solution to the transport equations in (QG) on $[0,T]$ with initial data $f_0$ and $g_0$ and forcing $a_L$, $a_\nu$ if for all $\tilde{\Omega}$ compactly contained in $\Omega$ and smooth test functions $\phi(t,x,y,z)$ compactly supported in $[-1, T+1]\times\tilde{\Omega}\times[-1,h+1]$ 
\begin{align*}
-\int_0^T \int_{\tilde{\Omega}\times[0,h]} &\left(\left( \partial_t \phi + \grad^\perp \Psi \cdot \grad \phi  \right)\left( \mathcal{L}(\Psi) + \beta_0 y \right) + \phi a_L  \right)\,dx\,dy\,dz\,dt \\
& \qquad = \int_{\tilde{\Omega}\times[0,h]} \phi|_{t=0}f \,dx\,dy\,dz
\end{align*}
and
\begin{align*}
\int_0^T \int_{\tilde{\Omega}\times\{0,h\}} \left(\left( \partial_t \phi + \grad^\perp \Psi \cdot \grad \phi  \right) \partial_{\nu}\Psi + \phi a_\nu  \right) \,dx\,dy\,dt = -\int_{\tilde{\Omega}\times\{0,h\}} \phi|_{t=0}g \,dx\,dy
\end{align*}
\end{definition}
We can now state our existence result.  
\begin{theorem}\label{main}
Let $(f_0,g_0,j_0)$ and $(a_L,a_\nu)$ satisfy \cref{compatibility}.  Then there exists a global weak solution $\Psi$ to (QG) such that 
\begin{enumerate}
\item $\mathcal{L}(\Psi)|_{t=0} = f_0$, $\dnu|_{t=0}=g_0$ and $\Psi$ satisfies \cref{weaktransport} for any $T>0$
\item There exists $c(t,z)$ such that for almost every time $t>0$, $\Psi(t)|_{\partial\Omega\times[0,h]} = c(t,z)$ 
\item For all $t>0$, $\grad\Psi(t)\cdot\nu_s \in L^2(\partial\Omega\times[0,h])$. If $j_0\in H^\frac{1}{2}(0,h)$, then
 $$ \int_{\partial\Omega\times\{z\}}  \grad\Psi(t)\cdot \nu_s \,d\omega = j_0(z), $$
with the equality holding pointwise in $z$.
\item For all time $t$, $\left(\mathcal{L}(\Psi)(t), \dnu(t), \grad\Psi\cdot\nu_s(t)\right)$ satisfies the compatibility condition in \cref{compatibility}
\item For all  $T>0$ and $t\in[0,T]$, $\Psi$ satisfies the bound
\begin{align*}
\| \mathcal{L}(\Psi)(t) \|&_{L^2(\Omega\times[0,h])} + \| \partial_{\nu}\Psi(t) \|_{L^2(\Omega\times\{0,h\})} + \| \nabla \Psi(t) \|_{H^\frac{1}{2}(\Omega\times[0,h])}  \\
&\leq C(\Omega,h,\lambda) \left( \| f \|_{L^2}+ \| g \|_{L^2} + \| j \|_{L^2} + \| a_L \|_{L^1\left([0,T];L^2\right)} + \| a_\nu \|_{L^1\left([0,T];L^2\right)}  \right). 
\end{align*}\end{enumerate}
\end{theorem}
\subsection{Inviscid Geostrophic Flows}\label{thesection}
Mathematical inquiry into (QG) is by now quite extensive.  Beale and Bourgeois \cite{bb} and Desjardins and Grenier \cite{dg} provided derivations of the 3D system from primitive equations.  As we are concerned with the inviscid model, our derivation follows that of Beale and Bourgeois. Puel and the second author \cite{pv} proved global existence results for initial data $\Psi_0$ such that $\mathcal{L}(\Psi_0), \nabla\Psi_0 \in L^2(\Rthreeplus)$, $\dnu_0 \in L^2(\mathbb{R}^2)$.  The first author \cite{novackweak} extended this result to initial data belonging to non-Hilbert Lebesgue spaces and identified the critical regularity at which the system conserves energy.

Study of the closely related surface quasi-geostrophic equation was initiated by Constantin, Majda, and Tabak \cite{cmt}.  To obtain SQG from (QG), one simplifies the model by assuming that $\lambda(z) \equiv 1$, $\beta_0 =0$, $a_L \equiv a_\nu \equiv 0$, and
$$ \Delta\Psi|_{t=0}=0. $$
As a result, $\Delta\Psi(t)\equiv 0$ uniformly in time, and the entire dynamic is encoded in the equation for $\theta = -\partial_z\Psi|_{z=0} = \half \Psi$
\begin{align}
\partial_t \theta + \mathcal{R}^\perp \theta \cdot \grad \theta = 0. \label{sqg} 
\end{align}
Resnick proved global existence of weak solutions for initial data in $L^2(\mathbb{T}^2)$ \cite{Resnick}.  Marchand extended Resnick's result to initial data belonging to $L^p(\mathbb{R}^2)$ or $L^p(\mathbb{T}^2)$ for $p>\frac{4}{3}$ \cite{Marchand}.  Both the proofs of Resnick and Marchand are based on a reformulation of the nonlinear term using a C\'{a}lderon commutator. 

To study \eqref{sqg}, it is common to add a dissipative term $\fractional \theta$.  The case $\alpha = \frac{1}{2}$ is physical and comes from considering viscous effects which produce Ekman layers at the boundary.  In the critical case $\alpha = \frac{1}{2}$, global regularity is known by different methods.  Proofs are given by Kiselev, Nazarov, and Volberg \cite{knv}, Caffarelli and the second author \cite{cv}, Constantin and Vicol \cite{cvicol}, and Kiselev and Nazarov \cite{Kiselev2010}.  Using the De Giorgi technique from \cite{cv} in combination with a bootstrapping argument and an appropriate Beale-Kato-Majda type criterion, the authors established global regularity for the full 3D system with critical dissipation in \cite{novackvasseur}. Buckmaster, Shkoller, and Vicol used the method of convex integration to show that one may prescribe any positive smooth profile for the Hamiltonians of both inviscid and dissipative SQG \cite{2016arXiv161000676B}.

The techniques used to produce weak solutions by Resnick and Marchand were adapted to bounded domains in a series of papers. In these works the Riesz transform on a bounded domain $\Omega$ is defined spectrally using eigenfunctions of the homogenous Dirichlet laplacian. First, Constantin and Ignatova \cite{ci}, \cite{ci2} proved nonlinear bounds and commutator estimates for the fractional laplacian and showed the existence of global weak solutions as well as derived interior regularity estimates for \eqref{sqg} with added critical dissipation in bounded domains. Constantin and Nguyen \cite{cn}, \cite{cn2} then showed the existence of global weak solutions of \eqref{sqg} in bounded domains as well as local and global strong solutions for supercritical and critical/subcritical versions of \eqref{sqg}, respectively.  

The weak solutions we construct \textit{cannot coincide} in general with solutions to \eqref{sqg} constructed using the spectral Riesz transform. The difference lies in the boundary conditions \eqref{bdrycond1} and \eqref{bdrycond2}.  At each time $t$, we reconstruct $\Psi$ by solving the elliptic problem
\[
\begin{dcases}
       \mathcal{L}(\Psi) = f &  \Omega\times[0,h] \\
       \partial_{\nu} \Psi = g & \Omega \times \{0,h\}\\
       \Psi(x,y,z) = c(z) \qquad & \partial\Omega \times \{z>0\}\\
       \int_{\partial\Omega\times\{z\}}  \grad \Psi\cdot \nu_s =j_0(z) & [0,h].
       \end{dcases}
\]
In particular, we do not require that the stream function $\Psi$ vanishes uniformly on the lateral boundary. While we consider the case of finite height $h$, the boundary conditions we impose would apply in the case of infinite height as well, which is the most common setting for SQG.

Conversely, let $\{ e_n \}$ be the orthonormal basis of eigenfunctions with corresponding eigenvalues $\{\lambda_n\}$ for the homogenous Dirichlet laplacian $-\lap_\Omega$ on $\Omega$, and let 
$$ \theta = \sum_n a_n(t) e_n(x,y) $$
be a solution to \eqref{sqg} posed on the bounded domain $\Omega$. Then the stream function $\Psi|_{z=0}$ is given by 
$$ \Psi|_{z=0} = \left( -\lap_\Omega \right) ^{-\frac{1}{2}} \theta = \sum_n a_n(t) \lambda_n^{-\frac{1}{2}} e_n(x,y),$$
and the harmonic extension for $z\in[0,\infty)$ is given by 
$$ \Psi(t,x,y,z)  = \sum_n a_n(t) e^{-z\sqrt{\lambda_n}} \lambda_n^{-\frac{1}{2}} e_n(x,y). $$
With this definition, $\Psi$ vanishes uniformly on $\partial\Omega\times[0,\infty)$.  In addition, if one were to impose \eqref{bdrycond2} on a solution to \eqref{sqg}, then integrating by parts in $(x,y)$ and passing the integral inside the sum gives
$$ \sum_n a'_n(t)  e^{-z\sqrt{\lambda_n}} \lambda_n^\frac{1}{2} \left( \int_{\Omega} e_n(x,y) \,dx\,dy \right) =0 $$
for all $z>0$.  One can see that this is only satisfied if 
$$ a'_n(t)  \left( \int_{\Omega} e_n(x,y) \,dx\,dy \right) =0 $$
for all $n$ and $t>0$, which cannot hold for any bounded domain $\Omega$ and initial data.  
The outline of this paper is as follows; in \cref{derivation}, we recall the derivation of the system from primitive equations while accounting for the impermeability. In \cref{ellipticsection}, we produce a solution to the stationary elliptic problem associated to the operator $\mathcal{L}$ and prove an elliptic regularity theorem for the solution. Finally, in \cref{mainsection}, we construct global weak solutions to (QG).

\section{Derivation from Primitive Equations}\label{derivation}
\subsection{Primitive Equations and Re-Scalings}
We begin from the so-called primitive equations following the derivation of Bourgeois and Beale \cite{bb}.  These equations represent the geostrophic balance, which is the balance of the pressure gradient with the Coriolis force.  The Boussinesq approximation has been made; that is, changes in density are ignored except when amplified by the effect of gravity.  After a re-scaling of the equations, a parameter which varies inversely with the speed of the rotation of the earth called the Rossby number shall appear.  Then performing a perturbation expansion in the Rossby number $\epsilon$ will yield the stratified system and boundary conditions \eqref{bdrycond1} and \eqref{bdrycond2}. Given a smooth, bounded set $\Omega \subset \mathbb{R}^2$ and a fixed height $h$, the following equations (after rescaling) will be posed on the cylindrical domain
$$ \Omega\times[0,h] . $$
We use the notation $\frac{D}{Dt}=\partial_t + \vec{u}\cdot \nabla$ for the material derivative, and the Coriolis force $\C=2\Theta \sin(\theta)$, where $\Theta$ is the angular velocity of the Earth and $\theta$ is the latitude.  Here $(u,v,w)$ is the fluid velocity, $p$ is the pressure and $\rho$ is the variation in density from a known background density profile $\bar{\varrho}(z)$.  That is, the density $\varrho$ satisfies
$$ \varrho = \bar{\varrho}(z) + \rho(x,y,z,t). $$
We further assume that the density is decreasing in $z$ and that $-\rho_z$ is bounded above and below away from zero. Throughout, we assume throughout that the fluid velocity is tangent to the boundary.

The primitive equations then are
\[
\begin{dcases}
 \frac{Du}{Dt} - \C v = -p_x \\
 \frac{Dv}{Dt} + \C u = -p_y \\
\frac{Dw}{Dt} +\rho g = -p_z \\
\nabla\cdot u =0 \\
\frac{D\varrho}{Dt} = 0. \\ 
\end{dcases}
\]
We rescale the equations in such a way so as to remove solutions which vary on a fast time scale.  Therefore, we set
$$ t=\frac{L}{U}t', \qquad u=Uu', \qquad (x,y,z)=L(x',y',z'). $$
Letting $\theta_0$ be a central latitude, we estimate $\C$ using the linear $\beta$-plane approximation by $$\C=2\Theta \sin(\theta_0) + 2\Theta \cos(\theta_0)(\theta-\theta_0):=\C_0 + 2\Theta \cos(\theta_0)(\theta-\theta_0).$$  The Rossby number $\epsilon$ is equal to $\frac{U}{\C_0 L}$. Set $\beta_0 = \frac{\cot(\theta_0)}{\epsilon}\frac{L}{r_0}$. We then have that 
\begin{align*}
\mathcal{C}&=2\Theta \sin(\theta_0) + 2\Theta \cos(\theta_0)(\theta-\theta_0) \\
&= \mathcal{C}_0(1+\epsilon \beta_0 y').
\end{align*}
We assume that $\frac{L}{r_0}$ is $O(\epsilon)$, allowing us to keep the factor of $\epsilon$ in front of $\beta_0$ even as $\epsilon\rightarrow 0$.  We scale the density variation by
$$ \rho = \frac{\C_0 U}{g} \rho' = \frac{U^2}{\epsilon L g} \rho' $$
and the reference density by 
$$ \bar{\varrho} = \frac{U^2}{\epsilon^2 L g} \bar{\varrho}' \ $$
This allows us to write the density non-dimensionally as
$$ \varrho = \frac{U^2}{\epsilon^2 L g}(\bar{\varrho}'(z)+ \epsilon \rho' ) $$
Finally, we scale the pressure by $p= \C_0 UL p'$.  Applying the scalings to the primitive equations, we obtain 
\[ \begin{dcases}
\frac{Du'}{Dt'} - \frac{1}{\epsilon}(1+\epsilon \beta_0 y')v' = -\frac{1}{\epsilon}p'_{x'} \\
\frac{Dv'}{Dt'} + \frac{1}{\epsilon}(1+\epsilon \beta_0 y')u' = -\frac{1}{\epsilon}p'_{y'} \\
\frac{Dw'}{Dt'} +\frac{1}{\epsilon}\rho'  = -\frac{1}{\epsilon}p'_{z'} \\
\nabla\cdot u' =0 \\
\frac{D\rho'}{Dt'} + \frac{1}{\epsilon} w'\bar{\varrho}'_{z'} = 0 .\\
\end{dcases}
\]
Let us abuse notation and drop the primes on our scaled equations.  Assume that the expansions 
$$ \vec{u} = \vec{u}(\epsilon) = \vec{u}^{(0)} + \epsilon \vec{u}^{(1)} + O(\epsilon^2)$$
and 
$$ \rho = \rho(\epsilon) = \rho^{(0)} + \epsilon \rho^{(1)} + O(\epsilon^2) $$
hold. Plugging this ansatz in, we obtain the zero-order equations
$$ v^{(0)} = p_x^{(0)}, \qquad u^{(0)} = -p_y^{(0)}, \qquad \rho^{(0)} = -p_z^{(0)}, \qquad w^{(0)}=0.$$
The last equation follows from the first two equations, the incompressibility (which gives that $w_z^{(0)}=0$), and the assumption that $w^{(0)}\equiv 0$ on the top and bottom of the domain. 

We move now to the first order equations.  Let us introduce the notation
$$ d_g = \partial_t - p_y^{(0)} \frac{\partial}{\partial_x} + p_x^{(0)} \frac{\partial}{\partial_y} $$
for the zero order geostrophic material derivative.  The first order equations are then
\[\begin{dcases}
d_g(-p_y^{(0)}) - v^{(1)} - \beta_0 y p_x^{(0)} = -p_{x}^{(1)} \\
d_g(p_x^{(0)}) + u^{(1)} - \beta_0 y p_y^{(0)} = -p_{y}^{(1)} \\
\rho^{(1)} = -p_{z}^{(1)} \\
\nabla \cdot u^{(1)}=0 \\
d_g(-p_z^{(0)}) + w^{(1)}\varrho_{z} = 0 .
\end{dcases}
\]
Let us divide the last equation by $-\frac{1}{\varrho_z}$.  We introduce the notation 
$$ \tilde{\nabla} = (\partial_x, \partial_y, -\frac{1}{\varrho_z}\partial_z). $$
Then we can consolidate the first order equations as
\begin{align}
d_g(\tilde{\nabla}p^{(0)}) + \beta_0(p^{(0)},0,0)^t &= (-p_y^{(1)}, p_x^{(1)},0)^t - (u^{(1)},v^{(1)}, w^{(1)})^t \nonumber \\
&\qquad  - \beta_0y(-p_y^{(0)}, p_x^{(0)}, 0)^t+ \beta_0(p^{(0)},0,0)^t. \label{gradienteqn}
\end{align}
Note that the right-hand side is divergence free and has no vertical component on the top and bottom boundaries of the domain. 
\subsection{Transporting $\mathcal{L}(\Psi)$ and $\dnu$}
We now take the divergence of \eqref{gradienteqn} in order to arrive at (QG).  As noted, the divergence of the right hand side is zero.  The divergence of $ \beta_0(p^{(0)},0,0)^t$ is $\beta_0 p_x^{(0)}$.  Examining the transport term  $d_g(\tilde{\nabla}p^{(0)})$ and calculating $\partial_z$ of the third component, we obtain
$$d_g(\partial_z(e_3 \cdot \tilde{\nabla}p^{(0)}))+ \partial_z u^{(0)}\partial_x (e_3 \cdot \tilde{\nabla}p^{(0)}) + \partial_z v^{(0)}\partial_y (e_3 \cdot \tilde{\nabla}p^{(0)}).$$
Using the fact that $u^{(0)}=-p_y^{(0)}$ and $v^{(0)}=p_x^{(0)}$, the second two terms cancel each other out.  The horizontal divergence $(\partial_x, \partial_y, 0)$ of $ d_g(\tilde{\nabla}p^{(0)})$ is easy to calculate from the stratification and the divergence free nature of the zero-order flow. We arrive at the equation 
$$ \left( \partial_t - p_y^{(0)}\partial_x + p_x^{(0)}\partial_y \right)  \left( p_{xx}^{(0)} + p_{yy}^{(0)} + (\lambda p_z^{(0)})_z + \beta_0 y \right) = 0 $$
after absorbing the $\beta$-plane term into the material derivative and defining $\lambda = -\frac{1}{\varrho_z}$.  Note that by the assumptions on the density, there exists $\Lambda$ such that $\frac{1}{\Lambda} \leq \lambda \leq \Lambda$. We shall use the notation $\Psi$ for the stream function $p^{(0)}$, allowing us to rewrite the system in the familiar form
\begin{align}\label{transportinside}
\left( \partial_t + \grad^\perp \Psi\cdot\grad\right) \left( \mathcal{L}(\Psi) + \beta_0 y \right)=0 .
\end{align}
Consider now the top and bottom $\Omega \times \{0\}$ and $\Omega \times \{h\}$.  Let $\nu$ denote the unit normal vector on the top and bottom. Considering the equation
$$d_g(-p_z^{(0)}) + w^{(1)}\varrho_{z} = 0,$$
using that $w^{(1)}\equiv 0$ on the top and bottom, and substituting the notation $\Psi$ for the stream function, we obtain
\begin{align}\label{transportboundary}
\left( \partial_t + \grad^\perp \Psi\cdot\grad\right) (\partial_{\nu}\Psi) =0.
\end{align}
\subsection{The Lateral Boundary}\label{bdrysection}
Now consider the sides $\partial\Omega \times [0,h]$ equipped with a horizontal normal vector $\nu_s$.  First, the impermeability requires that $\grad^\perp p^{(0)}\cdot \nu_s = 0$, implying that $p^{(0)}$ is constant on $\partial\Omega\times\{z\}$. Recalling that the stream function $\Psi=p^{(0)}$, we have that 
\begin{align}\label{constantonboundary}
\Psi(t,y,x,z)|_{\{\partial\Omega\times[0,h]\}}=c(t,z)
\end{align}
for some unknown function $c(t,z)$.

Let us next take the dot product of \eqref{gradienteqn} with $\nu_s$.  Due to the impermeability of the boundary, $$(u^{(1)}, v^{(1)}, w^{(1)})^t \cdot \nu_s =0. $$  
In addition, 
$$ (-p_y^{(1)}, p_x^{(1)},0)^t \cdot \nu_s = -(p_x^{(1)}, p_y^{(1)},0)^t \cdot \tau $$
where $\tau$ is the positively oriented tangent vector perpendicular to $\nu_s$.  Then we integrate around the boundary $\partial\Omega\times\{z\} \subset \partial\Omega\times [0,h]$ at a fixed height $z$.  Since $(p_x^{(1)}, p_y^{(1)},0)^t$ is a conservative vector field, 
$$ \int_{\partial\Omega\times\{z\}} (p_x^{(1)}, p_y^{(1)},0)^t \cdot \tau \,d\omega = 0 .$$
Notice that 
$$ \beta_0y(-p_y^{(0)}, p_x^{(0)}, 0)^t - \beta_0(p^{(0)},0,0)^t$$
is also the two-dimensional curl $\grad^\perp$ of the scalar field $-\beta_0 yp^{(0)}$.  Then we have that 
$$ \grad^\perp(-\beta_0 yp^{(0)}) \cdot \nu_s = \grad (\beta_0 yp^{(0)}) \cdot \tau .$$
As this is also a conservative vector field, the integral of this term around the boundary vanishes as well.  Thus we are left with
\begin{align}\label{remainder}
\int_{\partial\Omega\times\{z\}} (d_g \tilde{\nabla}p^{(0)})\cdot \nu_s \,d\omega = -\int_{\partial\Omega\times\{z\}} (\beta_0 p^{(0)},0,0)\cdot \nu_s \,d\omega.
\end{align}
Using \eqref{constantonboundary} shows that
$$ -\int_{\partial\Omega\times\{z\}} (\beta_0 p^{(0)},0,0)\cdot \nu_s \,d\omega $$
is zero.  Substituting in the stream function notation and applying the divergence theorem to the nonlinear term on the left hand side of \eqref{remainder}, we have that
\begin{align}
\int_{\partial\Omega\times\{z\}}  &\left(-p_y^{(0)}\partial_x \grad p^{(0)} + p_x^{(0)}\partial_y \grad p^{(0)}\right) \cdot \nu_s \,d\omega = \int_{\partial\Omega\times\{z\}} \grad\cdot \left( \grad^\perp\Psi \cdot \grad\grad\Psi \right) \cdot \nu_s \,d\omega \nonumber\\
&\qquad\qquad = \int_{\Omega\times\{z\}} \grad\grad^\perp \Psi : \grad\grad\Psi \,dx\,dy + \int_{\Omega\times\{z\}} \grad^\perp\Psi \cdot \grad \lap \Psi \,dx\,dy \nonumber\\
&\qquad\qquad =\int_{\Omega\times\{z\}} \grad\cdot\left( \grad^\perp\Psi \lap \Psi  \right) \,dx\,dy  \nonumber\\
&\qquad \qquad = \int_{\partial\Omega\times\{z\}} \lap \Psi \left(\grad^\perp\Psi \cdot \nu_s\right) \,d\omega \nonumber\\
&\qquad \qquad =0. \nonumber
\end{align}
Utilizing once again the notation $\Psi$ for the stream function, \eqref{remainder} therefore becomes
\begin{align}\label{averageneumann}
\frac{\partial}{\partial t}\int_{\partial\Omega\times\{z\}} ( \grad\Psi)\cdot \nu_s \,d\omega =0.
\end{align}
Collecting \eqref{transportinside}, \eqref{transportboundary}, \eqref{constantonboundary}, and \eqref{averageneumann}, we have formally derived the following system:
\[\begin{dcases}
\left( \partial_t + \grad^\perp \Psi \cdot \grad \right) \left( \mathcal{L}(\Psi) + \beta_0 y \right)=0 & \Omega\times[0,h] \\
        \left( \partial_t + \grad^\perp \Psi \cdot \grad \right) (\partial_{\nu}\Psi) =0 & \Omega\times\{0,h\}\\
      \frac{\partial}{\partial t}\int_{\partial\Omega\times\{z\}} ( \grad\Psi)\cdot \nu_s \,d\omega =0 & [0,h]\\
      \Psi=c(t,z) & \partial\Omega\times[0,h].
       \end{dcases}
\]
\section{The Elliptic Problem}\label{ellipticsection}
\subsection{Building a solution in $L^2$}
In order to show global existence of weak solutions to the time-dependent problem, we first solve the stationary elliptic problem which is transported by the fluid velocity $\grad^\perp\Psi$.   The elliptic operator is given by $\mathcal{L}$. The boundary conditions for the elliptic problem will be mixed in nature. We first impose a Neumann condition on the top and bottom of $\Omega\times[0,h]$ coming from the transport equation for $\dnu$. The condition that $$\Psi(t,x,y,z)|_{\partial\Omega\times[0,h]}=c(t,z)$$ will be structured into the Hilbert space within which we solve the elliptic problem.  Finally, the equation 
$$ \frac{\partial}{\partial t}\int_{\partial\Omega\times\{z\}} \grad \Psi\cdot \nu_s \,d\omega =0$$
means that 
$$ \int_{\partial\Omega\times\{z\}} \grad \Psi(t)\cdot \nu_s \,d\omega =\int_{\partial\Omega\times\{z\}} \grad \Psi(0)\cdot \nu_s \,d\omega =:j(z) $$
is determined from the initial data, and thus will be incorporated into the data of the elliptic problem. We now provide a weak formulation of this condition for (QG).
\begin{definition}\label{weakbdrycond2}
Let $T>0$ be given and $\Psi(t,x,y,z):[0,T]\times\Omega\times[0,h]$ be such that $\nabla\Psi, \mathcal{L}(\Psi)\in L^\infty([0,T];L^2(\Omega\times[0,h]))$, and for each time, $\Psi$ has mean value zero. Then we say that $\Psi$ satisfies \eqref{bdrycond2} weakly if there exists $j_0(z):[0,h]\rightarrow\mathbb{R}$ such that for each compactly supported smooth function $\phi(t,z): [0,T]\times[0,h]\rightarrow\mathbb{R}$,
$$\int_0^T\int_{\Omega\times[0,h]} \mathcal{L}(\Psi) \phi(t,z) - \Psi \partial_z \left( \lambda \partial_z \phi(t,z) \right)  \,dx\,dy\,dz\,dt= \int_0^T\int_0^h \phi(t,z) j_0(z) \,dz\,dt. $$
\end{definition}
An integration by parts shows that for smooth functions of time and space, \eqref{bdrycond2} is equivalent to \cref{weakbdrycond2}.  Indeed, 
\begin{align*}
\int_0^T\int_{\Omega\times[0,h]} &\mathcal{L}(\Psi) \phi(t,z) - \Psi \partial_z \left( \lambda \partial_z \phi(t,z) \right) \,dx\,dy\,dz\,dt\\
&= \int_0^T\int_{\Omega\times[0,h]} \mathcal{L}(\Psi) \phi(t,z) - \partial_z \left( \lambda \partial_z \Psi \right) \phi(t,z) \,dx\,dy\,dz\,dt\\
& = \int_0^T \int_{\Omega\times[0,h]} \left(\partial_{xx}\Psi + \partial_{yy}\Psi\right) \phi(t,z) \,dx\,dy\,dz\,dt\\
& = \int_0^T \int_0^h \int_{\partial\Omega} \phi(t,z) \grad\Psi\cdot \nu_s \,d\omega\,dz\,dt
\end{align*}

Thus we consider the elliptic problem for the unknown function $u$ with data $f:\Omega\times[0,h] \rightarrow \mathbb{R}$, $g:\Omega\times\{0,h\}\rightarrow \mathbb{R}$, and $j:[0,h]\rightarrow \mathbb{R}$. 
\begin{equation*}
(E) = 
\begin{mcases}[ll@{\ }l]
       \mathcal{L}(u) = f &  \Omega\times[0,h] & \quad(E1)\\
       \partial_{\nu} u = g & \Omega \times \{0,h\} & \quad(E2)\\
       u(x,y,z) = c(z) & \partial\Omega \times [0,h]& \quad(E3)\\
       \int_{\partial\Omega\times\{z\}}  \grad u\cdot \nu_s =j(z) & [0,h]& \quad(E4). 
\end{mcases}
\end{equation*}
Let us remark that to formulate $(E)$ variationally, it is not necessary for $(f,g,j)$ to satisfy the compatibility condition \cref{compatibility}.  Indeed our construction of approximate solutions will introduce a small error in the condition of \cref{compatibility} which will vanish in the limit. Thus when we say that $u$ is a solution to $(E)$, we generally mean it in the variational sense of $(V)$ (see \eqref{V} below).  If in addition, $(f,g,j)$ satisifes the compatibility condition so that $(V)$ is equivalent to $(E)$, we shall make note of this. To solve $(V)$ we require a specially constructed Hilbert space.
\begin{definition}
Define $H$ by
$$ H := \{ \alpha \in C^\infty\left(\bar{\Omega}\times[0,h]\right): \quad \int_{\Omega\times[0,h]}\alpha \,dx\,dy\,dz=0, \quad  \alpha|_{\partial\Omega \times [0,h]}(x,y,z)=\alpha(z) \}.$$
Using the notation $\tilde{\nabla}=(\partial_x,\partial_y, \lambda(z) \partial_z)$, equip $H$ with the inner product
$$ \langle \alpha, \gamma \rangle_\Hilbert := \int_{\Omega\times[0,h]} \tilde{\nabla} \alpha \cdot {\nabla}{\gamma} \,dx\,dy\,dz. $$
Define the Hilbert space $\Hilbert$ as the closure of $H$ under the norm induced by this inner product.
\end{definition}
By standard trace inequalities and Poincar\'{e}'s inequality, we have that for $\gamma\in\Hilbert$
\begin{align}
\| \gamma \|_{H^\frac{1}{2}(\partial(\Omega\times[0,h]))} \leq C(\Omega,h) \left( \| \gamma \|_{L^2(\Omega\times[0,h])} + \| \nabla \gamma \|_{L^2(\Omega\times[0,h])} \right) \leq C(\Omega,h,\lambda) \| \gamma \|_\Hilbert \label{trace}
\end{align}
We define a bilinear form $B(\alpha,\gamma):\Hilbert\times\Hilbert\rightarrow\mathbb{R}$ and functional $F(\gamma):\Hilbert\rightarrow\mathbb{R}$ by
$$ B(\alpha,\gamma) = \int_{{\Omega}\times[0,h]} \tilde{\nabla} \alpha \cdot {\nabla} \gamma \,dx\,dy\,dz$$
and 
$$ F(\gamma) = -\int_{{\Omega\times[0,h]}}{f \gamma}\,dx\,dy\,dz +  \int_{\Omega\times\{0,h\}} \lambda g \gamma \,dx\,dy + \int_0^h j(z) \gamma|_{\partial_\Omega\times\{z\}}\,dz.$$
The coercivity and continuity of the bilinear form $B$ is immediate from the assumptions on $\lambda(z)$ and the definition of $\Hilbert$.  In addition, we have that
\begin{align}
|F(\gamma)| &\leq \|f\|_{L^2(\Omega\times[0,h])} \| \gamma \|_{L^2(\Omega\times[0,h])} + \| \lambda \|_{L^\infty(0,h)} \| g \|_{L^2(\Omega\times\{0,h \})} \| \gamma \|_{L^2(\Omega\times\{0,h\})} \nonumber\\
&\qquad\qquad + \| j \|_{\left(H^\frac{1}{2}(\partial\Omega\times[0,h])\right)^*} \| \gamma \|_{H^\frac{1}{2}(\partial\Omega\times[0,h])}\nonumber\\
&\leq C(\Omega,h,\lambda) \left( \| f \|_{L^2} + \| g \|_{L^2} + \|  j \|_{\left(H^{\frac{1}{2}}\right)^*} \right) \| \gamma \|_\Hilbert \label{laxm}
\end{align}
after applying H\"{o}lder's inequality and \eqref{trace}.
Applying the Lax-Milgram theorem, we obtain a unique solution $u\in\Hilbert$ to the variational problem
\begin{align}\label{V}
B(u, \gamma) = F(\gamma)  \qquad \forall \gamma \in \Hilbert.  \qquad \qquad (V) 
\end{align}
Let us rigorously state the results of the above argument.

\begin{lemma}\label{elliptic}
For any data $f\in L^2(\Omega\times[0,h])$, $g \in L^2(\Omega\times\{0,h\})$, and $j \in \left(H^{\frac{1}{2}}([0,h])\right)^*$ there exists a unique solution $u \in \Hilbert$ to the variational problem $(V)$
with $$\|u\|_{\Hilbert} \leq C(\Omega,h,\lambda) \left( \| f \|_{L^2} + \| g \|_{L^2} + \|  j \|_{\left(H^{\frac{1}{2}}\right)^*} \right).$$
If in addition $(f,g,j)$ verifies the compatibility condition in \cref{compatibility}, then
\begin{enumerate}
\item $(E1)$ is satisfied in the weak sense
\item $(E2)$ is satisfied in the weak sense
\item $(E3)$ is satisfied pointwise
\item $(E4)$ is satisfied weakly. That is, for $\phi\in C^\infty$ depending only on $z$, $$\int_{\Omega\times[0,h]} \mathcal{L}(u) \phi(z) - u \partial_z \left( \lambda \partial_z \phi(z) \right) \,dx\,dy\,dz = \langle j, \phi \rangle $$ 
where $\langle \cdot, \cdot \rangle$ denotes duality between $\left(H^\frac{1}{2}\right)^*$ and $H^\frac{1}{2}$.
\end{enumerate}
\end{lemma}
\begin{proof}
The first claim is simply the above construction of $u$ as the solution to the variational problem $(V)$. For (1)-(4), the compatibility condition implies that constant functions $\gamma$ can be used in the weak formulation, and therefore any $C^\infty$ test function such that $\gamma(x,y,z)|_{\partial\Omega\times[0,h]}=c(z)$ is valid in the weak formulation. Parts (1) and (2) then follow from considering test functions which vanish on the lateral boundary $\partial\Omega\times[0,h]$.  Part (3) is a consequence of constructing the solution within $\Hilbert$.  Finally, (4) follows from noticing that when $\phi$ depends only on $z$, 
$$ B(u, \phi) = \int_{\Omega\times[0,h]} \lambda \partial_z u \partial_z \phi. $$
Rearranging the equality $B(u,\phi) = F(\phi)$ and using (1) finishes the proof.
\end{proof}

\subsection{Higher Regularity}
In order to build weak solutions, the operator which sends a triple  $(f,g,j)$ to the solution of the variational problem $(V)$ must map compactly into $\Hilbert$.  This will be achieved by proving an elliptic regularity theorem which asserts that the solution has strictly more than one derivative in $L^2(\Omega\times[0,h])$.  The proof is split up into four preliminary lemmas which correspond to isolating the effects of the compatibility condition, $g$, $f$, and $j$ on the regularity of the solution. Specifying a triple of data which does not satisfy \cref{compatibility} produces a solution by projecting, in an appropriate sense, the data onto the set of compatible data. Analysis of the effect $g$ is direct because solutions to the extension problem on bounded domains $\Omega$ can be written down explicitly.  Once the Neumann derivative has been removed, we analyze the effects of $f$ and $j$ by reflecting the solution over the boundaries $z=0,h$ and utilizing the standard difference quotient technique for elliptic regularity. Each step is proved for the special case $\lambda(z)\equiv 1$, i.e. when $\mathcal{L}=\Delta$. The four lemmas are combined in the proof of the following theorem, where we then provide a description of how to adapt the techniques to general smooth $\lambda$. 
\begin{theorem}\label{superduperelliptic}
Let $f\in L^2(\Omega\times[0,h])$, $g\in L^2(\Omega\times\{0,h\})$, and $j\in L^2([0,h])$.  Let $u\in \Hilbert$ be the unique variational solution to $(V)$ guaranteed by \cref{elliptic}.  Then 
$$ \|\nabla u\|_{H^\frac{1}{2}(\Omega\times[0,h])} \leq C(\Omega,h,\lambda)\left( \|f\|_{L^2(\Omega\times[0,h])} + \|g\|_{L^2(\Omega\times\{0,h\})}+\|j\|_{L^2([0,h])} \right).$$ 
\end{theorem}

Before beginning the analysis, we set several notations.  Let $\{e_n\}_{n=1}^\infty$ and $\{\lambda_n\}_{n=1}^\infty$ be the sequence of eigenfunctions and corresponding eigenvalues for the operator $-\lap$ on $\Omega$ with homogenous Dirichlet boundary conditions; that is, 
\[
\begin{dcases}
       -\lap e_n = \lambda_n e_n &  (x,y)\in\Omega \\
       e_n = 0 & (x,y)\in\partial\Omega. \\
       \end{dcases}
\]
For $s\geq 0$, define
$$ \bar{H}^s(\Omega) = \{ g=\sum_n g_n e_n \in L^2(\Omega): \sum_n \left(\lambdan\right)^s g_n e_n \in L^2(\Omega) \}.$$
By duality, we have that $$\left(\bar{H}^s(\Omega)\right)^* \cong \{ \{g_n\}_{n=1}^\infty\subset \mathbb{R}: \sum_n \frac{1}{\left(\lambdan\right)^{2s}} g_n^2 < \infty \}.$$
Real interpolation of Hilbert spaces $H_1, H_2$ is defined in the classical way (following the book of Bergh and Lofstrom for example \cite{berghlofstrom}).
For non-integer $s\in (-\infty, \infty)$, the Stein-Weiss interpolation theorem (see for example the book of Bergh and Lofstrom \cite{berghlofstrom}) gives that 
$$ [\bar{H}^{s_1}(\Omega), \bar{H}^{s_2}(\Omega)]_\theta = \bar{H}^s $$
for $s=\theta s_1 + (1-\theta)s_2$ where $s_1,s_2\in \mathbb{Z}$. When $s=0$, $\bar{H}^s(\Omega)$ coincides with $L^2(\Omega)$.  In general, $\bar{H}^s(\Omega)\subset H^s(\Omega)$ if $H^s(\Omega)$ is defined classically (see for example Constantin and Nguyen \cite{cn2}).

For $s\in(0,1)$, the fractional Sobolev spaces $H^s(\Omega\times[0,h])$ are defined by 
$$ H^s(\Omega\times[0,h]):= \left\lbrace h \in L^2(\Omega\times[0,h]): \frac{|h(x_1)-h(x_2)|}{|x_1-x_2|^{\frac{3}{2}+s}} \in L^2\left((\Omega\times[0,h])\times(\Omega\times[0,h])\right) \right\rbrace. $$
For $s\in \mathbb{N}+(0,1)$, $H^s(\Omega\times[0,h])$ is the subset of $L^2(\Omega\times[0,h])$ for which $$ \frac{|\nabla^{\left \lfloor{s}\right \rfloor }( h(x_1)-h(x_2))|}{|x_1-x_2|^{\frac{3}{2}+s-\left \lfloor{s}\right \rfloor}} \in L^2\left((\Omega\times[0,h])\times(\Omega\times[0,h])\right). $$
Classical interpolation results (see for example the work of Triebel \cite{triebelbook}, \cite{triebelpaper}) give that 
$$ [{H}^{s_1}(\Omega\times[0,h]), {H}^{s_2}(\Omega\times[0,h])]_\theta = H^{s}(\Omega\times[0,h])$$
for $s=\theta s_1 + (1-\theta)s_2$.

\begin{lemma}[Effect of the Compatibility Condition]
Let a triple $(f,g,j)$ with $f\in L^2(\Omega\times[0,h])$, $g\in L^2(\Omega\times\{0,h\})$, $j\in L^2(0,h)$ be given. Let $u$ be the solution to the variational problem with data $(f,g,j)$. Then there exists a constant $c$ depending only on 
$$\int_{\Omega\times[0,h]} f,\quad \int_{\Omega\times\{0,h\}} g,\quad \int_0^h j$$
such that $\Delta u = f + c$ and 
$$ |c| \lesssim \|f\|_{L^2} + \|g\|_{L^2} + \|j\|_{L^2}. $$
\end{lemma}
\begin{proof}
We define an operator $A: L^2(\Omega\times[0,h])\times L^2(\Omega\times\{0,h\})\times L^2(0,h) \rightarrow \mathbb{R} $ which maps a triple $(f,g,j)$ to a constant $c=A(f,g,j)$.  Since $\Hilbert$ only contains test functions with mean value zero, given $(x_1,y_1,z_1), (x_2,y_2,z_2) \in \Omega\times[0,h]$, choose a sequence of test functions which is the difference between two sequences of approximate identities centered at $(x_1,y_1,z_1), (x_2,y_2,z_2)$. Using this sequence of test functions in the variational formulation gives that $\Delta u(x_1,y_1,z_1)-\Delta u (x_2,y_2,z_2) = f(x_1,y_1,z_1)-f(x_2,y_2,z_2)$. Therefore, $\Delta u$ is equal to $f$ up to a constant $c$, and thus $A(f,g,j)=c$ is well-defined.  By the linearity of the variational problem, $A$ is linear.  To show that $A$ depends only on the integrals of $f$, $g$, and $j$, let $\bar{f}, \bar{g}, \bar{j}$ be given, each with mean value zero. Then $A(\bar{f},\bar{g},\bar{j})$ satisfies the compatibility condition, implying $\Delta \bar{u} = \bar{f}$ in a weak sense, and $A(\bar{f},\bar{g},\bar{j})=0$. Therefore $A$ depends only on 
$$\int_{\Omega\times[0,h]} f,\quad \int_{\Omega\times\{0,h\}} g,\quad \int_0^h j.$$
Now $A$ is a linear map from $\mathbb{R}^3\rightarrow \mathbb{R}$, and is therefore bounded.  That is, 
$$ |A(f,g,j)|^2 \lesssim \big{|}\int_{\Omega\times[0,h]} f\big{|}^2 + \big{|} \int_{\Omega\times\{0,h\}} g\big{|}^2 + \big{|} \int_0^h j\big{|}^2.$$
Applying H\"{o}lder's inequality finishes the proof.
\end{proof}

\begin{lemma}[Effect of $g$]\label{effectofg}
Consider the equation
\[
\begin{dcases}
       \Delta u = 0 &  \Omega\times[0,h] \\
       \partial_{\nu} u = g & \Omega \times \{0,h\}\\
       u = 0 \qquad & \partial\Omega \times [0,h].\\
       \end{dcases}
\]
for $g\in\bar{H}^s(\Omega\times\{0,h\})$, $s\geq -\frac{1}{2}$. Then there exists a solution $u$ which satisfies
$$ \| \nabla u \|_{H^{s+\frac{1}{2}}(\Omega\times[0,h])}  \leq C(\Omega,h) \|g\|_{\bar{H}^s(\Omega\times\{0,h\})}  $$
\end{lemma}
\begin{proof}
We begin by assuming that $g$ is smooth so that all calculations with higher derivatives are valid.  For arbitrary $g \in \bar{H}^s$, the claim follows from density of smooth functions. By assumption on $g$, there exist sequences of real numbers $\{ t_n \} , \{ b_n \}$ such that 
$$ g(0,x,y) = \sum_n b_n e_n(x,y), \qquad  g(h,x,y) = \sum_n t_n e_n(x,y) .$$
Define $$ {u} = \sum_n \left\lbrace \frac{t_n}{\lambdan} \frac{\cosh(z\lambdan)}{\sinh(h\lambdan)} + \frac{b_n}{\lambdan}\frac{\cosh((z-h)\lambdan)}{\sinh(h\lambdan)} \right\rbrace e_n(x,y). $$
Using that $\sinh(0) = 0$ and $(\sinh)''=(\cosh)'=\sinh$, 
we have that 
$$ -\frac{\partial}{\partial z} {u} |_{z=0} = g|_{z=0} $$ and 
$$ \frac{\partial}{\partial z} {u} |_{z=h} = g|_{z=h}. $$
In addition, it is immediate that
$$ \partial_{zz} {u} = -\lap u, $$
and therefore $\Delta {u} \equiv 0$.
Since $$\cosh(z\lambdan) \approx \sinh(h\lambdan) \approx e^{z\lambdan}$$
as $n\rightarrow \infty$, we have that 
$$ \half u \approx \partial_z u \in L^\infty([0,h];L^2(\Omega)) \subset L^2(\Omega\times[0,h]).  $$
Using the well-known fact that $\bar{H}^1(\Omega)=H_0^1(\Omega)$, we have that 
$$ \| \nabla u \|_{L^2(\Omega\times[0,h])} \leq C(\Omega,h) \| g \|_{L^2(\Omega\times\{0,h\})},  $$
and thus $u$ is a well-defined function in $\Omega\times[0,h]$ which solves the desired equation.

To sharpen this bound and obtain higher regularity estimates, we split the sum into four pieces corresponding to the four pieces of 
$$ \cosh(z\lambdan) = \frac{e^{z\lambdan}+e^{-z\lambdan}}{2}, \qquad \cosh((z-h)\lambdan) = \frac{e^{(z-h)\lambdan}+e^{-(z-h)\lambdan}}{2}. $$
Define 
$$ \tilde{t}_n := \frac{t_n}{\sinh(h\lambdan)} e^{h\lambdan}, \qquad  \tilde{g}:= \sum_n \tilde{t}_n e_n $$
so that $$ \|\tilde{g}\|_{\bar{H}^s(\Omega)} \leq C(\Omega) \| g \|_{\bar{H}^s(\Omega\times\{0,h\})} .$$
Then
$$ \tilde{u} := \sum_n e^{(z-h)\lambdan} \frac{\tilde{t}_n}{\lambdan} $$ 
is the solution to 
\[
\begin{dcases}
       \Delta \tilde{u} = 0 &  \Omega\times(-\infty,h] \\
       \partial_{\nu} \tilde{u} = \tilde{g} & \Omega \times \{h\}\\
       \tilde{u}(x,y,z) = 0 \qquad & \partial\Omega \times (-\infty,h].
       \end{dcases}
\]
Now we can write that 
\begin{align}
\| \nabla \tilde{u} \|_{\bar{H}^{s+\frac{1}{2}}(\Omega\times(-\infty,h))}   &= \int_{\Omega\times[0,h]} \nabla \left( (-\lap)^{\frac{1}{2}(s+\frac{1}{2})} \tilde{u} \right) \cdot\nabla\left( (-\lap)^{\frac{1}{2}(s+\frac{1}{2})} \tilde{u} \right) \nonumber\\
&= \int_{\Omega\times\{h\}} \partial_\nu (-\lap)^{\frac{1}{2}(s+\frac{1}{2})} \tilde{u} (-\lap)^{\frac{1}{2}(s+\frac{1}{2})} \tilde{u}\nonumber \\
&= \|\partial_\nu \tilde{u}\|_{\bar{H}^{s}(\Omega)}\nonumber\\
&\leq C(\Omega,h)\| g \|_{\bar{H}^{s}(\Omega\times\{0,h\})} 
\end{align}
Arguing in a similar fashion for the other parts of the infinite sum, we conclude that 
$$ \| \nabla u \|_{\bar{H}^{s+\frac{1}{2}}(\Omega\times[0,h])} \leq C(\Omega,h) \| g \|_{\bar{H}^{s}(\Omega\times\{0,h\})}. $$
If $s+\frac{1}{2}\in \mathbb{N}$, noticing that $(\partial_z)^{s+\frac{1}{2}} u \approx (-\lap)^\frac{s+\frac{1}{2}}{2} u$, we have that 
\begin{align}\label{interpolationbound}
\|\nabla u \|_{{H}^{s+\frac{1}{2}}(\Omega\times[0,h])} \leq C(\Omega,h) \| g \|_{\bar{H}^{s}(\Omega\times\{0,h\})}.
\end{align}

As noted above, for non-integer $s\in (-\frac{1}{2}, \infty)$, the Stein-Weiss interpolation theorem gives that 
$$ [\bar{H}^{s_1}(\Omega), \bar{H}^{s_2}(\Omega)]_\theta = \bar{H}^s(\Omega) $$
for $s=\theta s_1 + (1-\theta)s_2$, and interpolation of Hilbert-Sobolev spaces on Lipschitz domains gives that 
$$ [{H}^{s_1+\frac{1}{2}}(\Omega\times[0,h]), {H}^{s_2+\frac{1}{2}}(\Omega\times[0,h])]_\theta = H^{s+\frac{1}{2}}(\Omega\times[0,h]).$$  Interpolation of \eqref{interpolationbound} then concludes the proof of the lemma.
\end{proof}

In the following two lemmas, we address the effects of $f$ and $j$. While the solutions we consider are only variational a priori, for the sake of clarity we write each PDE using classical notation rather than the variational form.

\begin{lemma}[Effect of $f$]\label{effectoff}
Let $u\in\Hilbert$ be a variational solution to 
\[
\begin{dcases}
       \Delta u = f &  \Omega\times[0,h] \\
       \partial_{\nu} u = 0 & \Omega \times \{0,h\}\\
       u(x,y,z) = c(z) \qquad & \partial\Omega \times [0,h]\\
       \int_{\partial\Omega\times\{z\}}  \grad u\cdot \nu_s =0 & [0,h].
       \end{dcases}
\]
for data $f\in L^2(\Omega\times[0,h])$.  Then 
$$ \| u \|_{H^2(\Omega\times[0,h])} \leq C(\Omega,h) \| f\|_{L^2(\Omega\times[0,h])}. $$
\end{lemma}
\begin{proof}
Formally, the assumptions that $$\partial_\nu u \equiv 0, \qquad \int_{\partial\Omega\times\{z\}} \grad u \cdot \nu_s \equiv 0$$
give
\begin{align*}
\int_{\Omega\times[0,h]} \nabla(\partial_z u) \cdot \nabla (\partial_z u) = - \int_{\Omega\times[0,h]} \partial_z u \Delta(\partial_z u) = \int_{\Omega\times[0,h]} \partial_{zz}u f,
\end{align*}
implying that
\begin{align}\label{dzzbound}
\| \partial_{zz}u \|_{L^2(\Omega\times[0,h])} \leq C\| f \|_{L^2(\Omega\times[0,h])}.
\end{align}
Regularity of $\lap u$ would then follow from the equality 
$$ \lap u = f + A(f,0,0) - \partial_{zz}u. $$
Then we can write that for fixed $z$,
\[ 
\begin{dcases}
       \lap u = f + A(f,0,0) - \partial_{zz}u &  \Omega\times\{z\} \\
       u =c(z) &  \partial\Omega\times\{z\}.\\
       \end{dcases}
\]
Applying classical elliptic regularity theory $z$ by $z$ shows then that $\partial_{xy}u, \partial_{xx}u, \partial_{yy}u   \in L^2(\Omega\times[0,h])$.  Thus it remains to rigorously show \eqref{dzzbound}. 

Define
\[ u_E(x,y,z) = 
\begin{dcases}
       u(x,y,z) &  z\in[0,h] \\
       u(x,y,-z) &  z\in[-h,0]\\
       \end{dcases}
\]
and define $f_E$ similarly.  Let $\eta(z)$ be a smooth cutoff function depending only on $z$ such that $\eta \equiv 1$ for all $z\in[-\frac{h}{2},\frac{h}{2}]$ and $\eta$ is compactly supported in $[-\frac{3h}{4}, \frac{3h}{4}]$.  Define the difference quotient operator
$$ T_\epsilon \phi = \frac{\phi(x,y,z+\epsilon)-\phi(x,y,z)}{\epsilon}. $$
Then we can write 
\begin{align*}
-\int_{\Omega\times[-h,h]} \nabla u_E \cdot \nabla \left( T_{-\epsilon}(\eta^2 T_\epsilon u_E) \right)  &= \int_{\Omega\times[-h,h]} \nabla(T_\epsilon u_E) \cdot \nabla(\eta^2 T_\epsilon u_E)\\
&= \int_{\Omega\times[-h,h]} \nabla(T_\epsilon u_E) \cdot \nabla(T_\epsilon u_E) \eta^2\\
&\qquad + \int_{\Omega\times[-h,h]} \nabla(T_\epsilon u_E) \cdot \nabla \eta (2\eta T_\epsilon u_E)\\
&\geq \int_{\Omega\times[-\frac{h}{2},\frac{h}{2}]} |\nabla(T_\epsilon u_E)|^2 + \int_{\Omega\times[-h,h]} \nabla(T_\epsilon u_E) \cdot \nabla \eta (2\eta T_\epsilon u_E).
\end{align*}
Rearranging, we have that 
\begin{align*}
\int_{\Omega\times[-\frac{h}{2}, \frac{h}{2}]} |\nabla(T_\epsilon u_E)|^2 &\leq -\int_{\Omega\times[-h,h]} \nabla(T_\epsilon u_E) \cdot \nabla \eta (2\eta T_\epsilon u_E) \\
&\qquad -\int_{\Omega\times[-h,h]} \nabla u_E \cdot \nabla \left( T_{-\epsilon}(\eta^2 T_\epsilon u_E) \right)\\
&:= I + II.
\end{align*}
Examining I, we have that
\begin{align}
I &\leq C(\eta) \|\nabla(T_\epsilon u_E)\|_{L^2(\Omega\times[-h, h])}  \|\nabla u_E\|_{L^2(\Omega\times[-h,h])}\nonumber\\
&\leq C(\eta) \left( \frac{1}{8}\| \nabla(T_\epsilon u_E)\|_{L^2(\Omega\times[-h,h])}^2 + 4 \| f \|^2_{L^2(\Omega\times[0,h])} \right) \label{boundI}
\end{align}
Moving to II, we have that 
\begin{align}
II &= \int_{\Omega\times[-h,h]} T_{-\epsilon}(\eta^2(T_\epsilon u_E)) f_E \nonumber\\
&\leq C(\eta) \left(\frac{1}{8}\| T_{-\epsilon}T_\epsilon u_E \|^2_{L^2(\Omega\times[-h, h])} + 4\|f\|^2_{L^2(\Omega\times[0,h])} \right) \label{boundII}
\end{align}
Combining \eqref{boundI} and \eqref{boundII} and repeating the argument but this time with a reflection over $z=h$, it follows that
$$  \int_{\Omega\times[0,h]}| \nabla(T_\epsilon u) |^2 + \int_{\Omega\times[0,h]}| T_{-\epsilon}(T_\epsilon u) |^2 \leq 100C(\eta) \|f\|_{L^2(\Omega\times[0,h])}.  $$
The uniformity of this inequality in $\epsilon$ allows us to pass to a weak limit as $\epsilon\rightarrow 0$ to conclude that 
$$\| \nabla(\partial_z u) \|_{L^2(\Omega\times[0,h])} \leq C(\Omega,\eta) \| f \|_{L^2(\Omega\times[0,h])}^2.  $$
Regularity of $\partial_{xx} u$, $\partial_{xy}u $, and $\partial_{yy}u$ follows as described before, finishing the proof of the lemma.
\end{proof}

\begin{lemma}[Effect of $j$]\label{effectofj}
Let $u\in\Hilbert$ be a variational solution to 
\[
\begin{dcases}
       \Delta u = 0 &  \Omega\times[0,h] \\
       \partial_{\nu} u = 0 & \Omega \times \{0,h\}\\
       u(x,y,z) = c(z) \qquad & \partial\Omega \times \{z>0\}\\
       \int_{\partial\Omega\times\{z\}}  \grad u\cdot \nu_s =j(z) & [0,h].
       \end{dcases}
\]
for data $j\in H^s([0,h])$, $s\in [-\frac{1}{2}, \frac{1}{2}]$.  Then 
$$ \| \nabla u \|_{H^{(s+\frac{1}{2})}(\Omega\times[0,h])} \leq C(\Omega,h) \| j\|_{H^s([0,h])}. $$
\end{lemma}
\begin{proof}
The case $s=-\frac{1}{2}$ is the content of \cref{elliptic}.  We shall prove the case $s=\frac{1}{2}$ by hand and deduce the intermediate cases by interpolation.  

The proof for $s=\frac{1}{2}$ follows closely that of \cref{effectoff}.  Define $u_E$ and $j_E$ on $[-h,h]$ by reflection as before,  and define $T_\epsilon$ and $\eta$ similarly as well. In addition, let $\phi_\epsilon(z)$ be a one dimensional, smooth, even mollifier supported on a ball of radius $\epsilon$ around $0$. Note $c'(z)$ is yet not well defined as $\nabla u$ only belongs to $L^2([0,h])$ for now.  However, $c$ should satisfy $c'(0)=c'(h)=0$, and we shall mollify our test function in $z$ to take advantage of this. Thus we choose our test function to be 
$$ \phi_\epsilon \ast T_{-\epsilon}\left(\eta^2 T_\epsilon (u_E\ast \phi_\epsilon)\right) $$
Then we can write 
\begin{align*}
-\int_{\Omega\times[-h,h]} \nabla u_E \cdot \nabla &\left( \phi_\epsilon \ast T_{-\epsilon}(\eta^2 T_\epsilon (u_E\ast \phi_\epsilon)) \right)  = \int_{\Omega\times[-h,h]} \nabla(T_\epsilon (u_E\ast \phi_\epsilon)) \cdot \nabla(\eta^2 T_\epsilon (u_E\ast \phi_\epsilon))\\
&= \int_{\Omega\times[-h,h]} \nabla(T_\epsilon (u_E\ast \phi_\epsilon)) \cdot \nabla(T_\epsilon (u_E\ast \phi_\epsilon)) \eta^2\\
&\qquad + \int_{\Omega\times[-h,h]} \nabla(T_\epsilon (u_E\ast \phi_\epsilon)) \cdot \nabla \eta (2\eta T_\epsilon (u_E\ast \phi_\epsilon))\\
&\geq \int_{\Omega\times[-\frac{h}{2},\frac{h}{2}]} |\nabla(T_\epsilon (u_E\ast \phi_\epsilon))|^2\\
&\qquad + \int_{\Omega\times[-h,h]} \nabla(T_\epsilon (u_E\ast \phi_\epsilon)) \cdot \nabla \eta (2\eta T_\epsilon (u_E\ast \phi_\epsilon)).
\end{align*}
Rearranging, we have that 
\begin{align*}
\int_{\Omega\times[-\frac{h}{2}, \frac{h}{2}]} |\nabla(T_\epsilon (u_E\ast \phi_\epsilon))|^2 &\leq -\int_{\Omega\times[-h,h]} \nabla(T_\epsilon (u_E\ast \phi_\epsilon)) \cdot \nabla \eta (2\eta T_\epsilon(u_E\ast \phi_\epsilon)) \\
&\qquad -\int_{\Omega\times[-h,h]} \nabla (u_E\ast \phi_\epsilon) \cdot \nabla \left( T_{-\epsilon}(\eta^2 T_\epsilon (u_E\ast \phi_\epsilon)) \right)\\
&:= I + II.
\end{align*}

Examining I, we have that
\begin{align}
I &\leq C(\eta) \|\nabla(T_\epsilon (u_E\ast \phi_\epsilon))\|_{L^2(\Omega\times[-h, h])}  \|\nabla (u_E\ast \phi_\epsilon)\|_{L^2(\Omega\times[-h,h])}\nonumber\\
&\leq C(\eta) \left( \frac{1}{8}\| \nabla(T_\epsilon (u_E\ast \phi_\epsilon))\|_{L^2(\Omega\times[-h,h])}^2 + 4 \| j \|^2_{L^2([0,h])} \right) \label{boundIa}
\end{align}

Before examining II, notice that due to the compact support of $\eta$ in $[-h,h]$, we can assume without loss of generality that $u_E\ast\phi_\epsilon|_{\partial\Omega\times[0,h]}$ and $j_E\ast\phi_\epsilon$ are smooth, compactly supported functions on $[-h,h]$ and therefore can be expanded in Fourier series with coefficients $\hat{u}(k)\hat{\phi}_\epsilon(k)$ and $\hat{j}(k)\hat{\phi}_\epsilon(k)$, respectively. Note also that since $T_\epsilon$ ignores constants, we can assume without loss of generality that $\hat{u}(0)=\hat{j}(0)=0$, ensuring that fractional laplacians (as Fourier multipliers) of $u_E$ and $j_E$ are well-defined on $[0,h]$. Furthermore, since $(c_E\ast\phi)'(z)$ vanishes at $0$, the reflected function $u_E\ast\phi(z)|_{\partial\Omega\times[0,h]}$ belongs to $H^2([-h,h])$. In addition, $|\hat{\phi}_\epsilon(k)|\leq 1$ for all $k$ and converges to 1 as $\epsilon\rightarrow 0$. Then we can write
\begin{align}
II &= \int_{-h}^h (\eta^2 T_\epsilon (u_E\ast\phi_\epsilon)) T_\epsilon(j_E\ast\phi_\epsilon) \nonumber\\
&=  \sum_{k=-\infty}^\infty \left(\widehat{(T_\epsilon j_E)}(k) \frac{1}{k^\frac{1}{2}}\right)\left( \widehat{(\eta^2 T_\epsilon u_E )}(k) k^\frac{1}{2} \right) \hat{\phi}_\epsilon(k)^2 \nonumber\\
&\leq C(\eta) \left(\frac{1}{8}\| T_\epsilon u_E \|^2_{H^\frac{1}{2}(\partial\Omega\times[-h, h])} + 4\|T_\epsilon (-\lap)^{-\frac{1}{4}}j_E\|^2_{L^2([-h,h])} \right) \nonumber\\
&\leq C(\eta,\Omega) \left(\frac{1}{8}\|\nabla( T_\epsilon u_E ) \|^2_{L^2(\Omega\times[-h, h])} + 4\| j_E\|^2_{H^{\frac{1}{2}}([-h,h])} \right) .\label{boundIIa}
\end{align}
The last line follows from applying \eqref{trace} to $T_\epsilon u_E$ and noticing that 
$$\| T_\epsilon (-\lap)^{-\frac{1}{4}}j_E \|_{L^2} \approx \| \half (-\lap)^{-\frac{1}{4}}j_E \|_{L^2} \approx \|j_E\|_{H^\frac{1}{2}} . $$
Combining \eqref{boundIa} and \eqref{boundIIa} and repeating the argument but this time with a reflection over $z=h$, it follows that
$$  \int_{\Omega\times[0,h]}| \nabla(T_\epsilon u\ast\phi_\epsilon) |^2 \leq 100C(\eta,\Omega) \|j\|^2_{H^\frac{1}{2}([0,h])}.  $$
The uniformity of this inequality in $\epsilon$ allows us to pass to a weak limit as $\epsilon\rightarrow 0$ to conclude that 
$$\| \nabla(\partial_z u) \|_{L^2(\Omega\times[0,h])} \leq C(\Omega,\eta) \| j \|_{H^\frac{1}{2}([0,h])}^2.  $$
Regularity of $\partial_{xx} u$, $\partial_{xy}u $, and $\partial_{yy}u$ follows as for \cref{effectoff}, finishing the case $s=\frac{1}{2}$.  The intermediate cases follow again from interpolation.  
\end{proof}

We can now prove \cref{superduperelliptic}. 
\begin{proof}[Proof of \cref{superduperelliptic}]
We begin with $\lambda\equiv 1$, in which case $(V)$ is given by
\[
\begin{dcases}
       \Delta u = f &  \Omega\times[0,h] \\
       \partial_{\nu} u = g & \Omega \times \{0,h\}\\
       u(x,y,z) = c(z) & \partial\Omega \times [0,h]\\
       \int_{\partial\Omega\times\{z\}}  \grad u\cdot \nu_s =j(z) & [0,h].
       \end{dcases}  \qquad (V)
\]
First, apply \cref{effectofg} to build a solution $u_1$ to 
\[
\begin{dcases}
       \Delta u_1 = 0 &  \Omega\times[0,h]\\
       \partial_{\nu} u_1 = g & \Omega \times \{0,h\}\\
       u_1(x,y,z) = 0 & \partial\Omega \times [0,h]\\
       \end{dcases}  
\]
which satisfies
\begin{align*}
\| \nabla u_1 \|_{H^\frac{1}{2}(\Omega\times[0,h])} \leq C(\Omega,h) \| g \|_{L^2(\Omega\times\{0,h\})}.
\end{align*}
Now choose $c_1$ such that $\tilde{u}_1 = u_1 + c_1$ has mean value zero on $\Omega\times[0,h]$; then 
\begin{align}\label{firstbound}
\|\nabla\tilde{u}_1\|_{H^\frac{1}{2}(\Omega\times[0,h])}=\|\nabla{u}_1\|_{H^\frac{1}{2}(\Omega\times[0,h])} \leq C(\Omega,h) \| g \|_{L^2(\Omega\times\{0,h\})}.
\end{align}
By the trace estimate \eqref{trace}, 
$$ {j}_1(z) := \int_{\partial\Omega\times\{z\}} \grad \tilde{u}_1 \cdot \nu_s$$
is well-defined in $L^2([0,h])$ and satisfies
$$ \|j_1\|_{L^2([0,h])} \leq C(\Omega,h) \|\nabla\tilde{u}_1\|_{H^\frac{1}{2}(\Omega\times[0,h])} \leq C(\Omega,h) \| g \|_{L^2(\Omega\times\{0,h\})}.$$
Therefore, $\tilde{u}_1$ is the unique variational solution to 
\[
\begin{dcases}
       \Delta \tilde{u}_1 = 0 &  \Omega\times[0,h] \\
       \partial_{\nu} \tilde{u}_1 = g & \Omega \times \{0,h\}\\
       \tilde{u}_1(x,y,z) = c_1 & \partial\Omega \times [0,h]\\
       \int_{\partial\Omega\times\{z\}}  \grad \tilde{u}_1\cdot \nu_s =j_1(z) & [0,h].
       \end{dcases}  
\]
Now define $u_2:= u - \tilde{u}_1$; $u_2$ is then the unique variational solution to 
\[
\begin{dcases}
       \Delta{u}_2 = f &  \Omega\times[0,h] \\
       \partial_{\nu} {u}_2 = 0 & \Omega \times \{0,h\}\\
       {u}_2(x,y,z) = c_2(z) & \partial\Omega \times [0,h]\\
       \int_{\partial\Omega\times\{z\}}  \grad {u}_2\cdot \nu_s =j(z)-j_1(z) & [0,h].
       \end{dcases} 
\]
Define $u_3$ to as the unique variational solution to 
\[
\begin{dcases}
       \Delta{u}_3 = f &  \Omega\times[0,h] \\
       \partial_{\nu} {u}_3 = 0 & \Omega \times \{0,h\}\\
       {u}_3(x,y,z) = c_3(z) & \partial\Omega \times [0,h]\\
       \int_{\partial\Omega\times\{z\}}  \grad {u}_3\cdot \nu_s =0 & [0,h]
       \end{dcases} 
\]
and $u_4$ as the unique variational solution to 
\[
\begin{dcases}
       \Delta{u}_4 = 0 &  \Omega\times[0,h] \\
       \partial_{\nu} {u}_4 = 0 & \Omega \times \{0,h\}\\
       {u}_4(x,y,z) = c_4(z) & \partial\Omega \times [0,h]\\
       \int_{\partial\Omega\times\{z\}}  \grad {u}_4\cdot \nu_s =j(z)-j_1(z) & [0,h]
       \end{dcases}  
\]
so that $u_2 = u_3+u_4$.  Applying \cref{effectoff} to $u_3$ and \cref{effectofj} to $u_4$, we conclude that 
\begin{align}\label{secondbound}
\| \nabla u_2 \|_{H^\frac{1}{2}(\Omega\times[0,h])} \leq C(\Omega,h) \left( \|f\|_{L^2(\Omega\times[0,h])} + \|j-j_1\|_{L^2([0,h])} \right).
\end{align}
Combining \eqref{firstbound} and \eqref{secondbound}, we conclude that
$$ \| \nabla u \|_{H^\frac{1}{2}(\Omega\times[0,h])} = \| \nabla (\tilde{u}_1 + u_2) \|_{H^\frac{1}{2}(\Omega\times[0,h])} \leq C(\Omega,h) \left( \|f\|_{L^2} + \|g\|_{L^2} + \|j\|_{L^2}  \right). $$

We now sketch a proof of how to adapt the argument for arbitrary smooth $\lambda$ satisfying $\frac{1}{\Lambda}<\lambda<\Lambda$. Let $\phi_1,\phi_2,\phi_3$ be smooth functions of $z$ such that 
$$ \phi_1+\phi_2+\phi_3\equiv 1 \qquad \forall z\in[0,h] $$
and 
$$ \phi_1\in C_c^\infty(-\delta,2\delta),\qquad \phi_2\in C_c^\infty(\delta,h-\delta),\qquad \phi_3\in C_c^\infty(h-2\delta,h+\delta) $$
for $\delta$ to be chosen later. Because the proofs of \cref{effectoff} and \cref{effectofj} rely only on the variational structure, the difference quotient technique applies as well to general elliptic operators in divergence form (see for example sections 6.3 or 8.3 of Evans \cite{evans}). Since $\partial_\nu(\phi_2 u) \equiv 0$, it follows that $\phi_2 u \in H^\frac{3}{2}(\Omega\times[\delta,h-\delta])$.  

We focus now on $\phi_1 u$; the argument for $\phi_3 u$ is similar. The goal is to perform a change of variables in $z$ such that the elliptic operator after changing variables is given by the standard Laplacian plus lower order terms depending on the change of variables.  By writing
$$ \partial_z ( \lambda \partial_z u) = \lambda \partial_{zz}u + \partial_z \lambda \partial_z u, $$
notice that we can absorb the first order term $\partial_z \lambda \partial_z u$ into the right hand side, which we rename $\tilde{f}$. Then consider the ordinary differential equation
\[
\begin{dcases}
      \theta'(z') = \sqrt[]{\lambda(\theta(z'))} & z\in[0,\delta'] \\
      \theta(0)=0.
       \end{dcases}  
\]
By the Cauchy-Lipschitz theorem, for $\delta'$ small enough there exists a unique smooth solution $\theta$ which, by the positivity of $\lambda$, is a bijection between $[0,\delta']$ and $[0,\theta(\delta')]$.  Choose $\delta<\frac{\theta(\delta')}{2}$. Then 
\begin{align*} 
\partial_{z'z'}(u(x,y,\theta(z'))) &= u_{33}(x,y,\theta(z'))(\theta'(z))^2 + u_3(x,y,\theta(z'))\theta''(z') \\
&= u_{33}(x,y,\theta(z'))\lambda(\theta(z')) + u_3(x,y,\theta(z'))\theta''(z').
\end{align*}
Absorbing the second term $u_3(x,y,\theta(z'))\theta''(z')$ into the right hand side, (up to the effect of the localization $\phi_1$) the elliptic equation becomes 
$$\lap(u\circ \theta) + \partial_{z'z'}(u\circ\theta)  = \tilde{f}\circ\theta - (u_3\circ\theta)\theta'',$$
and we can repeat the original argument to show that $\phi_1 u \in H^\frac{3}{2}(\Omega\times[0,2\delta])$. Repeating the argument for $\phi_3 u$ and summing finishes the proof.

\end{proof}

\section{Proof of \cref{main}}\label{mainsection}

\subsection{Approximate solutions}
First, we adjust the initial data and forcing terms.  Let $\eta_\epsilon$ be a standard $\mathbb{R}^3$ mollifier supported in a ball of radius $\epsilon$. Define the extension of $f$ to $\mathbb{R}^3$ by
\[ f_E(x,y,z) = 
\begin{dcases}
       f_0(x,y,z) &  (x,y,z)\in\Omega\times[0,h] \\
       0 &  \text{otherwise},\\
       \end{dcases}
\] 
and mollify by setting $f_\epsilon:= f_E \ast \eta_\epsilon$.  After similarly extending $a_L(t)$ to $\mathbb{R}^3$ and $g, a_\nu(t)$ to $\mathbb{R}^2\times\{0,h\}$ by zero and mollifying (time by time for the forcing terms), we obtain spatially smooth (for example $a_{L,\epsilon}\in L^1([0,T];C^k(\mathbb{R}^3))$ for any $k$) sequences of functions such that the following convergences hold:
$$ f_\epsilon \rightarrow f_0 \quad \text{in} \quad L^2(\Omega\times[0,h])  $$
$$ g_\epsilon \rightarrow g_0 \quad \text{in} \quad L^2(\Omega\times\{0,h\})  $$
$$ a_{L,\epsilon} \rightarrow a_L \quad \text{in} \quad L^1 \left([0,T];L^2(\Omega\times[0,h])\right)  $$
$$ a_{\nu,\epsilon} \rightarrow a_\nu \quad \text{in} \quad L^1 \left([0,T];L^2(\Omega\times\{0,h\})\right).  $$
We define the approximate (QG) solution operators $S_\epsilon: C\left([0,T];\Hilbert\right)\rightarrow C\left([0,T];\Hilbert\right)$ for $\epsilon>0$ in several steps.  These operators shall provide solutions to linear transport equations with mollified velocity fields.
\begin{enumerate}[label=Step \arabic*:]
\item Let $P\in C\left([0,T];\Hilbert\right)$, and let $c(z)$ be the lateral boundary values of $P$ as usual. We extend $P(t)$ to $\mathbb{R}^3$ for each time in a way which allows for a simple construction of a smooth, stratified velocity field from $\grad^\perp P$ which is supported in a small neighborhood of $\Omega\times[0,h]$. 
\[ P_e(x,y,z) = 
\begin{dcases}
       P(x,y,z) &  (x,y,z)\in\Omega\times[0,h] \\
       c(z) &  (x,y,z)\in \Omega^{\mathsf{C}}\times[0,h]\\
       \end{dcases}
\] 
and 
\[ P_E(x,y,z) = 
\begin{dcases}
       P_e(x,y,z) &  (x,y,z)\in\mathbb{R}^2\times[0,h] \\
       P_e(x,y,0) &  (x,y,z)\in\mathbb{R}^2\times[-\infty,0] \\
       P_e(x,y,h) &  (x,y,z)\in\mathbb{R}^2\times[h,\infty]. \\
       \end{dcases}
\] 
Mollify $P_E$ by setting $P_\epsilon := P_E \ast \eta_\epsilon$. 
\item Consider the transport equations for $F_\epsilon$ and $G_\epsilon$ given by 
\[
\begin{dcases}
        \left(\partial_t + \grad^\perp P_{\epsilon} \cdot \grad \right) \left( F_\epsilon + \beta_0 y\right) = a_{L,\epsilon} & \mathbb{R}^2\times[0,h]\times[0,\infty) \\
         \left(\partial_t + \grad^\perp P_{\epsilon}\cdot\grad\right)  G_\epsilon =a_{\nu,\epsilon} & \mathbb{R}^2\times\{0,h\}\times[0,\infty)\\
      F_\epsilon = f_\epsilon & t=0 \\
      G_\epsilon  = g_\epsilon & t=0. \\
       \end{dcases}  
\]
Since the initial data, forcing terms, and velocity fields are all smooth, we can produce global in time solutions $F_\epsilon$ and $G_\epsilon$ by the method of characteristics. Notice that $F_\epsilon$ and $G_\epsilon$ are defined for $(x,y)\in\mathbb{R}^2$ but supported in a neighborhood of order $\epsilon$ around $\Omega$.
\item At each time $t\geq 0$, apply \cref{elliptic} to define $Q_\epsilon(t)$ as the solution to 
\begin{align*}
B(Q_\epsilon(t), v) & := \int_{\Omega\times[0,h]} \tilde{\nabla} Q_\epsilon(t) \cdot \nabla v \\
&= -\int_{\Omega\times[0,h]} -F_\epsilon(t) v + \int_{\Omega\times\{0,h\}} \lambda G_\epsilon (t) v + \int_0^h v|_{\partial\Omega\times[0,h]} j_0.\\
&=: F(v)
\end{align*}
Define $S_\epsilon(P) := Q_\epsilon$. We remark that because $F_\epsilon$ and $G_\epsilon$ are defined as solutions to transport equations for $(x,y)\in\mathbb{R}^2$ rather than $\Omega$, the compatibility condition is lost.  However, \cref{elliptic} still produces a solution to the abstract variational problem, and we will recover the compatibility condition in the limit.  
\end{enumerate}

In search of fixed points, we will show that the operators $\{ S_\epsilon \}_{\epsilon>0}$ are compact, continuous operators from $C\left([0,T];\Hilbert\right)$ to itself with bounded range. Continuity of the operators results from examining the characteristics of the mollified transport equations, while the proof of compactness will require \cref{superduperelliptic} and the Aubin-Lions lemma. We split the argument into three lemmas.
\begin{lemma}[Continuity]\label{Scontinuity}
The operator $S_\epsilon$ is continuous from $C\left([0,T];\Hilbert\right)$ to itself, with modulus of continuity dependent on $\epsilon$.
\end{lemma}
\begin{proof}
Let 
$$ P_n \rightarrow P \quad \text{in} \quad C\left([0,T];\Hilbert\right). $$
Define $S_\epsilon(P_n):= Q_{n,\epsilon}$.  Using the notation from the construction of the operators $S_\epsilon$, let $F_{n,\epsilon}$ and $G_{n,\epsilon}$ be the solutions to the transport equations with mollified velocity fields $\grad^\perp P_{n,\epsilon}$. Applying \cref{elliptic}, for fixed $t\in[0,T]$, 
\begin{align*}
\| \left(Q_{n_1,\epsilon}-Q_{n_2,\epsilon} \right)(t) \|_{\Hilbert} &\lesssim \bigg{(} \|(F_{n_1,\epsilon}-F_{n_2,\epsilon})(t)\|_{L^2(\Omega\times[0,h])} + \|(G_{n_1,\epsilon}-G_{n_2,\epsilon})(t)\|_{L^2(\Omega\times\{0,h\})} \bigg{)}.
\end{align*}
Therefore, it suffices to show that 
\begin{align}\label{continuity}
\sup_{t\in[0,T]} \left\lbrace \|(F_{n_1,\epsilon}-F_{n_2,\epsilon})(t)\|_{L^2(\Omega\times[0,h])} + \|(G_{n_1,\epsilon}-G_{n_2,\epsilon})(t)\|_{L^2(\Omega\times\{0,h\})} \right\rbrace \rightarrow 0
\end{align}
as $n_1,n_2 \rightarrow \infty$. 

First, notice that due to the mollification, given $k\in\mathbb{N}$, there exist constants $C(\epsilon,k)$ depending on $\epsilon, k$ such that 
\begin{align}\label{mollifiedconvergence}
\| \grad^\perp( P_{n_1,\epsilon}-P_{n_2,\epsilon} )\|_{L^\infty\left([0,T];C^k(\Omega\times[0,h])\right)} \leq C(\epsilon,k) \|P_{n_1}-P_{n_2}\|_{C\left([0,T];\Hilbert\right)}.
\end{align}
Fix $(t,x,y,z)\in[0,T]\times\mathbb{R}^2\times[0,h]$, and let $\Gamma_{n_i}$ for $i=1,2$ solve
\[
\begin{dcases}
       \dot{\Gamma}_{n_i}(s) = \grad^\perp P_{n_i}\left(s,\Gamma_{n_i}(s)\right) & s\in [0,t] \\
       \Gamma_{n_i}(t)=(x,y,z) \\
       \end{dcases}    
\]
Then 
$$F_{n_i}(t,x,y,z) = f_\epsilon(\Gamma_{n_i}(t)) + \int_0^t a_{L,\epsilon}(\Gamma_{n_i}(s)) \,ds, $$
and 
$$F_{n_1}(t,x,y,z) - F_{n_2}(t,x,y,z)= f_\epsilon(\Gamma_{n_1}(t))  - f_\epsilon(\Gamma_{n_2}(t))+ \int_0^t a_{L,\epsilon}(\Gamma_{n_1}(s)) - a_{L,\epsilon}(\Gamma_{n_2}(s)) \,ds. $$
Applying \eqref{mollifiedconvergence} and using the smoothness of $f_\epsilon$ and $a_{L,\epsilon}$ shows that as $n_1,n_2 \rightarrow \infty$, $F_{n_1}(t,x,y,z) - F_{n_2}(t,x,y,z)$ converges to $0$ uniformly for $(t,x,y,z) \in [0,T]\times\mathbb{R}^2\times[0,h]$. Arguing similarly for $G_{n_1,\epsilon},G_{n_2,\epsilon}$ then shows \eqref{continuity}.
\end{proof}
\begin{lemma}[Time Derivative Bounds]\label{timederivativebounds}
Let $P\in C\left([0,T];\Hilbert\right)$ with mollified velocity field $P_\epsilon$, and put $S_\epsilon(P):=Q_\epsilon$. Consider $\tilde{\nabla}Q_\epsilon(t)$ as an element of $\Hilbert^*$ acting by the rule
$$ v \rightarrow \langle Q(t), v \rangle_\Hilbert \qquad \forall v \in \Hilbert. $$
Then $\partial_t \tilde{\nabla} Q_\epsilon$ is a bounded linear functional in $L^\infty([0,T];(\Hilbert\cap H^3(\Omega\times[0,h]))^*)$ and 
$$ \|\partial_t \tilde{\nabla}Q_\epsilon\|_{L^\infty\left([0,T];(\Hilbert\cap H^3(\Omega\times[0,h]))^*\right)} \leq C\left( f_0, g_0, a_L, a_\nu, \beta_0,h,\|\grad^\perp P_{\epsilon}\|_{L^\infty([0,T]\times[0,h];L^2(\Omega))} \right) . $$
\end{lemma}
\begin{proof}

The distributional time derivative of $\tilde{\nabla}Q_\epsilon(t)$ is defined by the equality
$$ \langle \partial_t \tilde{\nabla}Q_\epsilon, \phi \rangle  := -\int_0^T \phi'(t) \tilde{\nabla}Q_\epsilon(t) \,dt $$
for all $\phi \in C_c^\infty(0,T)$.  To show that $\partial_t \tilde{\nabla}Q_\epsilon(t) \in L^\infty([0,T];(\Hilbert\cap H^3(\Omega\times[0,h]))^*))$, we test against functions $v\in \Hilbert\cap H^3(\Omega\times[0,h])$. First, recall the definitions of $F_{\epsilon}$ and $G_{\epsilon}$ as the solutions to the linear transport equations with mollified velocity $\grad^\perp P_{\epsilon}$ as in Step 2. Then we have
\begin{align}
-\int_0^T \phi'(t) \langle  &Q_{\epsilon}(t), v \rangle_\Hilbert \,dt \nonumber = -\int_0^T  \int_{\Omega\times[0,h]} \tilde{\nabla} Q_{\epsilon}(t,x,y,z) \cdot \nabla v(x,y,z)\phi'(t) \,dx\,dy\,dz\,dt \nonumber \\
&= -\int_0^T B(Q_{\epsilon}(t), \phi'(t)v) \,dt \nonumber\\
&= -\int_0^T F(\phi'(t)v) \,dt \nonumber\\
&= -\int_0^T \bigg{(} -\int_{\Omega\times[0,h]} F_{\epsilon}(t,x,y,z) \phi'(t)v(x) \,dx\,dy\,dz \nonumber\\
&\qquad +\int_{\Omega\times\{0,h\}} \lambda G_{\epsilon}(t,x,y,z) \phi'(t)v(x,y,z) \,dx\,dy  + \int_0^h j_0(z)v|_{\partial\Omega}(z)\phi'(t)\,dz\bigg{)}\,dt \label{timederivative}
\end{align}
Since $F_{\epsilon}$ and $G_{\epsilon}$ are classical solutions to transport equations, we have that
\begin{align}
\int_0^T \int_{{\Omega}\times[0,h]} \left(\left( v \phi' + \grad^\perp P_{\epsilon} \cdot \grad v \phi  \right)( F_{\epsilon} + \beta_0 y ) + \phi v a_{L,\epsilon}  \right)  = 0 \label{transport1}
\end{align}
and
\begin{align}
\int_0^T \int_{{\Omega}\times\{0,h\}} \left(\left( v \phi' + \grad^\perp P_{\epsilon} \cdot \grad v \phi  \right) G_{\epsilon} + \phi v a_{\nu,\epsilon}  \right) =0 .  \label{transport2} 
\end{align}
Plugging \eqref{transport1} and \eqref{transport2} into \eqref{timederivative} and noticing that 
$$ \int_0^T \int_0^h j_0 v \phi' = - \int_0^T \int_0^h (j_0 v)' \phi = 0 $$
gives
\begin{align}
-\int_0^T &\phi'(t) \langle  Q_{\epsilon}(t), v \rangle_\Hilbert \,dt \nonumber = -\int_0^T \int_{{\Omega}\times[0,h]} \left(\left(\grad^\perp P_{\epsilon} \cdot \grad v \phi  \right)( F_{\epsilon} + \beta_0 y ) + \phi v a_{L,\epsilon}  \right)  \nonumber\\
& \qquad \qquad \qquad \qquad + \int_0^T \int_{{\Omega}\times\{0,h\}} \lambda \left(\left(  \grad^\perp P_{\epsilon} \cdot \grad v \phi  \right) G_{\epsilon} + \phi v a_{\nu,\epsilon}  \right) \nonumber\\
&\leq \|\grad^\perp P_{\epsilon}\|_{L^\infty([0,T]\times[0,h];L^2(\Omega))} \| \grad v \|_{L^\infty(\Omega\times[0,h])} \|\phi\|_{L^\infty(0,T)} \left(\|F_{\epsilon}\|_{L^\infty\left([0,T];L^2(\Omega\times[0,h])\right)}+ \beta_0 h \right) \nonumber\\
&\qquad+\|\phi\|_{L^\infty(0,T)}\|v\|_{L^\infty(\Omega\times[0,h])} \|a_{L,\epsilon}\|_{L^1([0,T];L^2(\Omega\times[0,h]))} \nonumber\\
&\qquad + \Lambda \|\grad^\perp P_{\epsilon}\|_{L^\infty([0,T]\times[0,h];L^2(\Omega))} \| \grad v \|_{L^\infty(\Omega\times\{0,h\})} \|\phi\|_{L^\infty(0,T)} \|G_{\epsilon}\|_{L^\infty([0,T];L^2(\Omega\times\{0,h\}))} \nonumber\\
&\qquad +\|\phi\|_{L^\infty(0,T)}\|v\|_{L^\infty(\Omega\times\{0,h\})} \|a_{\nu,\epsilon}\|_{L^1([0,T];L^2(\Omega\times[0,h]))} \nonumber\\
& \leq \| v \|_{H^3(\Omega\times[0,h])}\|\phi\|_{L^\infty(0,T)} \left(1+\|\grad^\perp P_{\epsilon}\|_{L^\infty([0,T]\times[0,h];L^2(\Omega))}\right) \times \nonumber\\
&\qquad\qquad \left( 1+\|a_{L}\|_{L^1([0,T];L^2(\Omega\times[0,h]))} +\|a_{\nu}\|_{L^1([0,T];L^2(\Omega\times[0,h]))}\right) \times   \nonumber \\
&\qquad\qquad\qquad \left(1+ \|f_0\|_{L^2(\Omega\times[0,h])}+ \beta_0 h  + \Lambda \|g_0\|_{L^2(\Omega\times\{0,h\})}   \right). \nonumber
\end{align}
\end{proof}

\begin{lemma}[Compactness]\label{compactness}
Let $\{\epsilon_n\}_{n=1}^\infty$ be a sequence of positive numbers, $P_n$ be a sequence of functions in $C([0,T];\Hilbert)$, and $S_{\epsilon_n}(P_n):= Q_n$. If there exists $M$ such that the mollified velocity fields $\grad^\perp P_{n,\epsilon_n}$ satisfy 
$$ \sup_n \| \grad^\perp P_{n,\epsilon_n}\|_{L^\infty([0,T]\times[0,h];L^2(\Omega))} < M $$
then up to a subsequence, there exists $Q\in C([0,T];\Hilbert)$ such that $Q_n$ converges strongly in $C([0,T];\Hilbert)$ to $Q$.
\end{lemma}
\begin{proof}
To set notation, $Q_{n}$ is the solution to the variational problem
$$ B_n(Q_{n}, v) = F_n(v), \qquad v \in \Hilbert $$
described in Step 3. Define the Banach spaces
$$ \mathcal{B}_1 = \Hilbert^*,\qquad  \mathcal{B}_0 = \Hilbert\cap H^\frac{3}{2}(\Omega\times[0,h]),\qquad \mathcal{B}_2 = \left( \Hilbert\cap H^3(\Omega\times[0,h]) \right)^*  $$
We set $u^*\in\Hilbert^*$ as the linear functional on $\Hilbert$ defined by $v \rightarrow \langle u,v \rangle_\Hilbert$. This identification provides an isometric linear bijection between $\Hilbert$ and $\Hilbert^*$. Then by the Rellich-Kondrachov theorem and the observed isomorphism, the embedding of $\mathcal{B}_0$ into $\mathcal{B}_1$ is compact. The inclusion map from $\mathcal{B}_1$ to $\mathcal{B}_2$ is continuous as well.  Applying \cref{elliptic}, invoking the isomorphism between $\Hilbert$ and $\Hilbert^*$, and using the divergence free property of the mollified transport equations, we have that $ Q_{n}^* \in C\left( [0,T]; \Hilbert^* \right)$, and for $t\in[0,T]$,
\begin{align}\label{thefirst}
\| Q_{n,\epsilon}^*(t) \|_{\Hilbert^*} \leq C(\Omega,h,\lambda) \left( \| f_0 \|_{L^2} + \| g_0 \|_{L^2} + \| j_0 \|_{L^2} +\| a_L \|_{L^1\left([0,T];L^2\right)} + \| a_\nu \|_{L^1\left([0,T];L^2\right)} \right).
\end{align}
In addition, \cref{superduperelliptic} provides the bound
\begin{align}\label{thesecond}
\| \tilde{\nabla}Q_{n}(t) \|_{H^\frac{1}{2}(\Omega\times[0,h])} \leq C(\Omega,h,\lambda) \left( \| f_0 \|_{L^2} + \| g_0 \|_{L^2} + \| j_0 \|_{L^2} +\| a_L \|_{L^1\left([0,T];L^2\right)} + \| a_\nu \|_{L^1\left([0,T];L^2\right)} \right),
\end{align}
showing that $Q_{n} \in L^\infty([0,T];\mathcal{B}_0)$.  By \cref{timederivativebounds} and the existence of the constant $M$, $\partial_t (Q_{n}^*)$ is a sequence of operators bounded in $L^\infty([0,T];\mathcal{B}_2)$, and the assumptions of the Aubins-Lions lemma are satisfied.  We have then that $Q_n^*$ is precompact in $C \left([0,T];\Hilbert^*\right)$, and thus $Q_{n}$ is precompact in $C \left([0,T];\Hilbert\right)$. 
\end{proof}
\begin{corollary}[Fixed Points]
Each operator $S_\epsilon$ has a fixed point $\Psi_\epsilon$.
\end{corollary}
\begin{proof}
\cref{Scontinuity} shows that $S_\epsilon$ is continuous.  By the mollification of the velocity fields, there exists $C(\epsilon)$ such that for all $P\in {C([0,T];\Hilbert)} $,
$$\|\grad^\perp P_{\epsilon}\|_{L^\infty([0,T]\times[0,h];L^2(\Omega))} \leq C(\epsilon) \| P \|_{C([0,T];\Hilbert)}. $$
Then by \cref{compactness} with $\epsilon_n=\epsilon$ for all $n$, $S_\epsilon$ is a compact operator. By \eqref{thefirst}, the range of $S_\epsilon$ is bounded. Therefore, we can apply the Leray-Schauder fixed point theorem (see Evans \cite{evans}) to obtain a fixed point $\Psi_\epsilon$.
\end{proof}
\subsection{Passing to the Limit}
Consider the sequence of fixed points $\Psi_\epsilon$ to the operators $S_\epsilon$. By definition, $S_\epsilon(\Psi_\epsilon) = \Psi_\epsilon$, and therefore $\Psi_\epsilon$ solves the variational problem
\begin{align*}
B_{\Psi_\epsilon}(\Psi_\epsilon(t), v) &= \int_{\Omega\times[0,h]} \tilde{\nabla} \Psi_\epsilon(t) \cdot \nabla v \\
&= -\int_{\Omega\times[0,h]} -F_\epsilon(t) v + \int_{\Omega\times\{0,h\}} \lambda G_\epsilon (t) v + \int_0^h v|_{\partial\Omega\times[0,h]} j_0.\\
&= F_{\Psi_\epsilon}(v)
\end{align*}
Let us extract a subsequence which we index by $n \in\mathbb{N}$ such that $F_{\epsilon_n}$ converges weakly to $F$ in $L^\infty\left([0,T];L^2(\Omega\times[0,h])\right)$, and $G_{\epsilon_n}$ converges weakly to $G$ in $L^\infty([0,T];L^2(\Omega\times\{0,h\}) )$. Define $\Psi(t)$ as the time by time solution to 
\begin{align}
B(\Psi(t), v) &= \int_{\Omega\times[0,h]} \tilde{\nabla} \Psi(t) \cdot \nabla v \nonumber \\
&= -\int_{\Omega\times[0,h]} -F(t) v + \int_{\Omega\times\{0,h\}} \lambda G(t) v + \int_0^h v|_{\partial\Omega\times[0,h]} j_0. \nonumber\\
&= F(v) \nonumber
\end{align}

Recall that in Step 1, $\Psi_{\epsilon_n}:\Omega\times[0,h]\rightarrow\mathbb{R}$ was extended to $\Psi_{\epsilon_n, E}:\mathbb{R}^3\rightarrow\mathbb{R}$ and then mollified at length scale ${\epsilon_n}$ to produce a smooth velocity field $\Psi_{\epsilon_n,E}\ast \eta_{\epsilon_n}$. The following technical lemma regarding both the convergence of $\Psi_{\epsilon_n}$ and the mollified velocity fields $\Psi_{\epsilon_n,E}\ast \eta_{\epsilon_n}$ shall be useful. 

\begin{lemma}\label{stupidconvergence}
\begin{enumerate}
\item Up to a subsequence, $\Psi_{\epsilon_n}$ converges strongly to $\Psi$ in $C([0,T];\Hilbert)$
\item For any compact subdomain $\tilde{\Omega} \subset \Omega$, $\grad^\perp\Psi_{\epsilon_n,E}\ast \eta_{\epsilon_n}$ converges strongly up to a subsequence to $\grad^\perp\Psi$ in $C([0,T];L^2(\tilde{\Omega}\times[0,h]))$.
\end{enumerate}
\end{lemma}
\begin{proof}
To show (1), we consider \cref{compactness} with $P_n=\Psi_{\epsilon_n}=Q_n$. By \eqref{thesecond}
\begin{align*}
\sup_n \| \tilde{\nabla}\Psi_{\epsilon_n} &\|_{L^\infty([0,T];H^\frac{1}{2}(\Omega\times[0,h]))} \leq  C(\Omega,h,\lambda)\times \\
&\qquad \left( \| f_0 \|_{L^2} + \| g_0 \|_{L^2} + \| j_0 \|_{L^2} +\| a_L \|_{L^1\left([0,T];L^2\right)} + \| a_\nu \|_{L^1\left([0,T];L^2\right)} \right).
\end{align*}
Taking the trace then shows that 
\begin{align}
\sup_n \|\grad^\perp \Psi_{\epsilon_n}&\|_{L^\infty([0,T]\times[0,h];L^2(\Omega))} \leq  C(\Omega,h,\lambda)\times\nonumber \\
&\qquad \left( \| f_0 \|_{L^2} + \| g_0 \|_{L^2} + \| j_0 \|_{L^2} +\| a_L \|_{L^1\left([0,T];L^2\right)} + \| a_\nu \|_{L^1\left([0,T];L^2\right)} \right). \label{ialsoneedthis}
\end{align}
By construction of the extension $\Psi_{\epsilon_n,E}$, 
\[ \grad^\perp\Psi_{\epsilon_n,E}(x,y,z) = 
\begin{dcases}
       \grad^\perp\Psi_{\epsilon_n}(z) & (x,y,z)\in \tilde{\Omega}\times[0,h] \\
       \grad^\perp\Psi_{\epsilon_n}(0) & (x,y,z)\in \tilde{\Omega}\times (-\infty,0]\\          \grad^\perp\Psi_{\epsilon_n}(h) & (x,y,z)\in \tilde{\Omega}\times [h,\infty), 
       \end{dcases}    
\]
showing that $\grad^\perp\Psi_{\epsilon_n, E}$ is uniformly bounded in $n$ in $L^\infty([0,T]\times[-\epsilon_n,h+\epsilon_n];L^2(\Omega))$.
Therefore, 
\begin{align}
\sup_{n} \|\grad^\perp\Psi_{\epsilon_n,E}(t)\ast\eta_{\epsilon_n}\|_{L^\infty([0,T]\times[0,h];L^2(\Omega))}< \infty \label{ineedthis}
\end{align}
Thus the assumptions of \cref{compactness} are satisfied, and up to a subsequence, $\Psi_{\epsilon_n}$ converges to $\Psi$ strongly in $C([0,T];\Hilbert)$.

Moving to (2), let $\tilde{\Omega}$ be a fixed compact subdomain of $\Omega$. We have that for $t\in[0,T]$,
\begin{align}
\limsup_{n\rightarrow\infty}\|&\grad^\perp(\Psi_{\epsilon_n,E}\ast\eta_{\epsilon_n}(t)-\Psi(t))\|_{L^2(\tilde{\Omega}\times[0,h])} \leq \limsup_{n\rightarrow\infty}\|\grad^\perp(\Psi_{\epsilon_n,E}\ast\eta_{\epsilon_n}(t)-\Psi_{\epsilon_n}(t))\|_{L^2(\tilde{\Omega}\times[0,h])}\nonumber\\
&\qquad\qquad\qquad\qquad\qquad + \limsup_{n\rightarrow\infty}\|\grad^\perp(\Psi_{\epsilon_n}(t)-\Psi(t))\|_{L^2(\tilde{\Omega}\times[0,h])}  \nonumber\\
&\leq \sup_{n} \|\grad^\perp\Psi_{\epsilon_n,E}\ast{\eta_{\epsilon_n}}(t)\|_{L^2(\tilde{\Omega}\times([0,\delta)\cup(h-\delta,h]))}+\sup_{n} \|\grad^\perp\Psi_{\epsilon_n}(t)\|_{L^2(\tilde{\Omega}\times([0,\delta)\cup(h-\delta,h]))}   \nonumber\\
&\qquad+\limsup_{n\rightarrow\infty}\|\grad^\perp(\Psi_{\epsilon_n,E}\ast\eta_{\epsilon_n}(t)-\Psi_{\epsilon_n}(t)) \|_{L^2(\tilde{\Omega}\times[\delta,h-\delta])}. \nonumber
\end{align}
By \eqref{ialsoneedthis} and \eqref{ineedthis}, the first two terms go to zero as $\delta\rightarrow 0$. So it suffices to show that for fixed $\delta$ that
$$ \limsup_{n\rightarrow\infty}\|\grad^\perp(\Psi_{\epsilon_n,E}\ast\eta_{\epsilon_n}(t)-\Psi_{\epsilon_n}(t))\|_{L^2(\tilde{\Omega}\times[\delta,h-\delta])} =0. $$
For $n$ large enough, $$\Psi_{\epsilon_n,E}\ast\eta_{\epsilon_n} = \Psi_{\epsilon_n}\ast\eta_{\epsilon_n} \qquad \forall (x,y,z)\in(\tilde{\Omega}\times[\delta,h-\delta]).$$
By extending $\Psi_{\epsilon_n}$ from $\tilde{\Omega}\times[\delta,h-\delta]$ to $\mathbb{R}^3$ using a standard Sobolev extension operator, it suffices to prove the claim for functions defined on all of $\mathbb{R}^3$.  Using the Fourier characterization of $H^\frac{3}{2}(\mathbb{R}^3)$, we can write 
\begin{align*}
\| \grad^\perp \Psi_{\epsilon_n}\ast\eta_{\epsilon_n}(t) - \grad^\perp\Psi_{\epsilon_n}(t)\|_{L^2(\mathbb{R}^3)}^2 &\leq \int_{\mathbb{R}^3} |\xi|^2|\hat{\Psi}_{\epsilon_n}(t,\xi)|^2 |\hat{\eta}_{\epsilon_n}(\xi)-1|^2 \,d\xi\\
&\leq \int_{\mathbb{R}^3} |\xi|^2 (1+|\xi|^2)^\frac{1}{2}|\hat{\Psi}_{\epsilon_n}(t,\xi)|^2 \frac{|\hat{\eta}(\epsilon_n\xi)-1|^2}{(1+|\xi|^2)^\frac{1}{2}} \,d\xi\\
&\leq \sup_{n} \|\Psi_{\epsilon_n}(t)\|_{H^\frac{3}{2}(\mathbb{R}^3)}^2 \cdot \sup_{\xi} \frac{|\hat{\eta}(\epsilon_n\xi)-1|^2}{(1+|\xi|^2)^\frac{1}{2}}.
\end{align*}
which goes to zero uniformly in $t$ as $n\rightarrow\infty$ since $\hat{\eta}$ is smooth and $\hat{\eta}(0)=1$, concluding the proof.
\end{proof}

We now pass to the limit to show that $\Psi$ is the solution we seek. As first utilized in \cite{pv} and then again in \cite{novackweak}, the strong convergence at the level of $\nabla\Psi_{\epsilon_n}$ and the reformulation of the system in terms of $\nabla\Psi_{\epsilon_n}$ give compactness in the nonlinear term of the reformulation. Fix a test function $\phi$ as in \cref{weaktransport}. Then
\begin{align}
-\int_0^T \int_{\tilde{\Omega}\times[0,h]} \left(\left( \partial_t \phi + \grad^\perp (\Psi_{\epsilon_n}\ast\eta_{\epsilon_n}) \cdot \grad \phi  \right) F_{\epsilon_n} + \phi a_{L,n}  \right) \,dx\,dy\,dz\,dt  = \int_{\tilde{\Omega}\times[0,h]} \phi|_{t=0}f_{\epsilon_n} \,dx\,dy\,dz \label{qg1}
\end{align}
and
\begin{align}
\int_0^T \int_{\tilde{\Omega}\times\{0,h\}} \left(\left( \partial_t \phi + \grad^\perp (\Psi_{\epsilon_n}\ast\eta_{\epsilon_n}) \cdot \grad \phi  \right) G_{\epsilon_n} + \phi a_{\nu,n}  \right) \,dx\,dy\,dt = -\int_{\tilde{\Omega}\times\{0,h\}} \phi|_{t=0}g_{\epsilon_n} \,dx\,dy \label{qg2}
\end{align}
For each time $t>0$, let $A_{\epsilon_n}(t)\in\Hilbert$ be the solution to 
\begin{align*}
B_{A,n}(A_{\epsilon_n}(t), v) &:= \int_{\Omega\times[0,h]} \tilde{\nabla} A_{\epsilon_n}(t) \cdot \nabla v \\
&= -\int_{\Omega\times[0,h]} -a_{L,\epsilon}(t) v + \int_{\Omega\times\{0,h\}} a_{\nu,\epsilon}(t)v \\
&= F_{A,n}(v)
\end{align*}
Using $\partial_t \phi + \grad^\perp(\Psi_{\epsilon_n}\ast\eta_{\epsilon_n})\cdot\grad\phi$ and $\phi$ as test functions in the variational formulations for $\Psi_{\epsilon_n}$ and $A_{\epsilon_n}$, respectively, turns \eqref{qg1} and \eqref{qg2} into
\begin{align}
-\int_0^T \int_{\tilde{\Omega}\times[0,h]} &\left(\left( \partial_t \nabla \phi + \grad^\perp(\Psi_{\epsilon_n}\ast\eta_{\epsilon_n}):\grad\nabla\phi \right) \cdot \tilde{\nabla} \Psi_{\epsilon_n}  + \nabla \phi \cdot \tilde{\nabla} A_{\epsilon_n} \right ) \,dx\,dy\,dz\,dt \nonumber \\
&= \int_{\tilde{\Omega}\times[0,h]} \nabla\phi|_{t=0}\cdot \tilde{\nabla}\Psi_{\epsilon_n}|_{t=0} \,dx\,dy\,dz \label{rqg}
\end{align}
Applying \cref{stupidconvergence}, we pass to the limit to obtain
\begin{align*}
-\int_0^T \int_{\tilde{\Omega}\times[0,h]} &\left(\left( \partial_t \nabla \phi + \grad^\perp\Psi:\grad\nabla\phi \right) \cdot \tilde{\nabla} \Psi  + \nabla \phi \cdot \tilde{\nabla} A \right ) \,dx\,dy\,dz\,dt \\
&= \int_{\tilde{\Omega}\times[0,h]} \nabla\phi|_{t=0}\cdot \tilde{\nabla}\Psi|_{t=0} \,dx\,dy\,dz
\end{align*}
Rearranging the variational formulation now for $\Psi$ gives
\begin{align*}
-\int_0^T \int_{\tilde{\Omega}\times[0,h]} \left(\left( \partial_t \phi + \grad^\perp \Psi \cdot \grad \phi  \right) F + \phi a_L  \right) \,dx\,dy\,dz\,dt = \int_{\tilde{\Omega}\times[0,h]} \phi|_{t=0}f \,dx\,dy\,dz 
\end{align*}
and
\begin{align*}
\int_0^T \int_{\tilde{\Omega}\times\{0,h\}} \left(\left( \partial_t \phi + \grad^\perp \Psi \cdot \grad \phi  \right) G + \phi a_\nu  \right) \,dx\,dy\,dt = -\int_{\tilde{\Omega}\times\{0,h\}} \phi|_{t=0}g\,dx\,dy . 
\end{align*}

The final part of the proof consists of showing that $\Psi(t)$ solves a variational problem for all $t\in[0,T]$ which verifies the compatibility condition \cref{compatibility}. By construction of the approximate solution operators, $\Psi_{\epsilon_n}(t)$ solves the variational problem with data 
$$ \left( F_{\epsilon_n}(t)|_{\Omega\times[0,h]} , G_{\epsilon_n}(t)|_{\Omega\times\{0,h\}}, j_0 \right). $$
In addition, $F_{\epsilon_n}(t)$ and $G_{\epsilon_n}(t)$ are supported in a neighborhood of order ${\epsilon_n}$ around $\Omega$ for all time $t\in[0,T]$. Then using the weak convergence of $F_n|_{\Omega\times[0,h]}$ to $F$ and $G_n|_{\Omega\times\{0,h\}}$ to $G$, we have
\begin{align*} 
\int_{\Omega\times[0,h]} F_n(t) + &\int_{\Omega\times\{0,h\}} G_n(t) + \int_0^h j_0 \rightarrow \\
&\int_{\Omega\times[0,h]} f_0 + \int_0^t \int_{\Omega\times[0,h]} a_L +\int_{\Omega\times\{0,h\}} g_0 +  \int_0^t \int_{\Omega\times\{0,h\}} a_\nu  + \int_0^h j_0
\end{align*}
Using the assumption that $(f_0,g_0,j_0)$ and $(a_L, a_\nu)$ satisfy \cref{compatibility} shows that $\Psi(t)$ solves an elliptic problem with compatible data.  Then by \cref{elliptic}, $\mathcal{L}(\Psi)=F$ and $\dnu=G$ in the traditional weak sense.

We have thus shown that $\Psi$ satisfies part (4) of \cref{main}, and therefore \cref{weaktransport} and part (1) of \cref{main}. For part (2), the choice of $\Psi$ as a weak limit of functions belonging to $L^\infty\left([0,T];\Hilbert\right)$ implies that $\Psi(t)\in\Hilbert$ for almost every $t$.  Therefore, $\Psi$ must depend only on $z$ on the lateral boundary, and there exists $c(t,z)$ such that $\Psi|_{\partial\Omega\times[0,h]}=c(t,z)$ for almost every time.  To show part (3), first note that in light of the $H^\frac{1}{2}\left(\Omega\times[0,h]\right)$ bound on $\nabla\Psi$, $\grad\Psi\cdot\nu_s(t)$ is well-defined in $L^2(\partial\Omega\times[0,h])$ for almost every time. Assuming now that $j_0\in H^\frac{1}{2}(0,h)$, let $\alpha_n(z)$ be a compactly supported smooth function in $(\frac{1}{n},h-\frac{1}{n})$ such that $\alpha_n(z)=1$ for all $z\in(\frac{2}{n},h-\frac{2}{n})$. Applying \cref{superduperelliptic} to $\alpha_n \Psi(t)$ shows that $\alpha_n \Psi(t) \in H^2(\Omega\times[0,h])$, and therefore $\lap\Psi(t,z)\in L^2(\Omega)$ for $z\in(\frac{2}{n},h-\frac{2}{n})$. Then
$$ \int_{\Omega\times\{z\}} \lap \Psi = \int_{\partial\Omega\times\{z\}} \grad\Psi\cdot\nu_s $$
and applying \cref{elliptic} shows part (3). Finally, the bounds in part (5) follow from the divergence free property of the flow and \cref{superduperelliptic}, completing the proof of the theorem.

\bibliography{references}
\bibliographystyle{plain}
\nocite{cv}
\nocite{triebelbook}
\nocite{triebelpaper}
\nocite{berghlofstrom}
\nocite{evans}
\nocite{2016arXiv161000676B}
\nocite{cvicol}
\nocite{knv}
\nocite{pv}
\nocite{cmt}
\nocite{Marchand}
\nocite{Resnick}
\nocite{dg}
\nocite{bb}
\nocite{novackvasseur}
\nocite{a}
\nocite{novackweak}
\nocite{ci}
\nocite{ci2}
\nocite{cn}
\nocite{cn2}
\end{document}